\documentclass[11pt]{amsart}
\usepackage{amsmath}
\usepackage[psamsfonts]{amssymb}
\usepackage{amsfonts}
\usepackage{algorithm2e}
\usepackage{graphicx}
\usepackage{verbatim}
\usepackage{dsfont}
\usepackage{tikz-cd}
\usepackage{mathrsfs}
\usepackage{color}
\usepackage[english]{babel}
\usepackage{fontenc}
\usepackage{indentfirst}
\usepackage{tikz}
\usetikzlibrary{matrix,arrows,decorations.pathmorphing}
\usepackage{algorithm2e}
\usepackage[all]{xy}
\usepackage[T1]{fontenc}
\usepackage{hyperref}
\usepackage{mathtools}
\usepackage{array}
\usepackage{tabularx}

\theoremstyle{definition}
\newtheorem{theorem}{Theorem}[section]
\newtheorem{prop}[theorem]{Proposition}
\newtheorem{lemma}[theorem]{Lemma}

\newtheorem{cor}[theorem]{Corollary}

\newtheorem{conj}[theorem]{Conjecture}

\newtheorem{ex}[theorem]{Example}

\newtheorem{dfn}[theorem]{Definition}

\newtheorem{remark}[theorem]{Remark}

\def\ep{\epsilon}

\def\R{\mathbb{R}}
\def\Z{\mathbb{Z}}
\def\G{\mathcal{G}}

\def\F{\mathcal{F}}
\def\H{\mathcal{H}}
\def\N{\mathbb{N}}

\def\K{\mathcal{K}}
\def\D{\mathcal{D}}

\def\T{\mathbb{T}}

\def\Q{\mathbb{Q}}

\def\k{{\bf k}}

\DeclarePairedDelimiter\ceil{\lceil}{\rceil}

\newcolumntype{P}[1]{>{\centering\arraybackslash}m{#1}}

\begin{document}
\title{Relative growth rate and contact Banach-Mazur distance} 

\noindent
\author{Daniel Rosen}
\email{daniel.rosen@rub.de}
\address{Fakult\"{a}t f\"{u}r Mathematik, Universit\"{a}tstr. 150, Ruhr-Universit\"{a}t Bochum, 44780 Bochum, Germany}

\author{Jun Zhang}
\email{jun.zhang.3@umontreal.ca}
\address{Department of Mathematics and Statistics, University of Montreal, C.P. 6128 Succ. Centre-Ville Montreal, QC H3C 3J7, Canada}

\date{\today}

\begin{abstract}
In this paper, we define a non-linear version of Banach-Mazur distance in the contact geometry set-up, called contact Banach-Mazur distance and denoted by $d_{\rm CBM}$. Explicitly, we consider the following two set-ups, either on a contact manifold $W \times S^1$ where $W$ is a Liouville manifold, or a closed Liouville-fillable contact manifold $M$. The inputs of $d_{\rm CBM}$ are different in these two cases. In the former case the inputs are (contact) star-shaped domains of $W \times S^1$, and in the latter case the inputs are contact 1-forms of $M$. In particular, the contact Banach-Mazur distance $d_{\rm CBM}$ defined in the former case is motivated by the concept, relative growth rate, which was originally defined and studied in \cite{EP00}. In addition, we investigate the relations of $d_{\rm CBM}$ to various numerical measurements in contact geometry and symplectic geometry, for instance, (contact) shape invariant defined in \cite{Eli91}, (coarse) symplectic Banach-Mazur distance recently defined and studied in \cite{SZ18} and \cite{Ush18}. Moreover, we obtain several large-scale geometric properties in terms of $d_{\rm CBM}$. Finally, we propose a quantitative comparison between elements in the derived categories of sheaves of modules (over certain topological spaces). This is based on several important properties of the singular support of sheaves. 
\end{abstract}

\maketitle

\tableofcontents

\section{Introduction and main results} \label{ssec-intro}
Quantitative comparisons between objects in symplectic geometry have been investigated for decades. For instance, in terms of the objects from dynamics, the well-known Hofer distance $d_{\rm Hofer}$ is defined between Hamiltonian diffeomorphisms of a symplectic manifold, which is based on the fundamental work \cite{Hof90} and \cite{LM95}. This led to an influential research direction called Hofer geometry (\cite{Pol01}), as well as far-reaching applications to Hamiltonian dynamic (\cite{Ush13}, \cite{PS16}, \cite{PSS17}, \cite{Zha19}). For another instance, inspired by \cite{OP17}, the recent work \cite{SZ18}, \cite{Ush18} and \cite{GU19} studied quantitative comparisons between symplectic objects constructed from a geometric perspective, that is, star-shaped domains of a Liouville manifold. These are symplectic Banach-Mazur distance $d_{\rm SBM}$ and coarse symplectic Banach-Mazur distance $d_c$, and they can be viewed as non-linear analogues of the classical Banach-Mazur distance in convex geometry (\cite{Rud00}). In the set-up of contact geometry, a Hofer-like distance $d_{\rm Shelukhin}$ between contactomorphisms or contact isotopies was recently defined and investigated in \cite{She17}. This opens new research directions (\cite{RZ18}), and also provides a useful measurement if one wants to study analogue quantitative problems in contact geometry (\cite{DRS16}). This paper is devoted to the study of quantitative comparisons between both contactomorphisms (different from $d_{\rm Shelukhin}$ in \cite{She17}) and ``contact domains'' of certain contact manifolds. The following table summarizes what was elaborated above. 

\begin{center}
\includegraphics[scale=0.85]{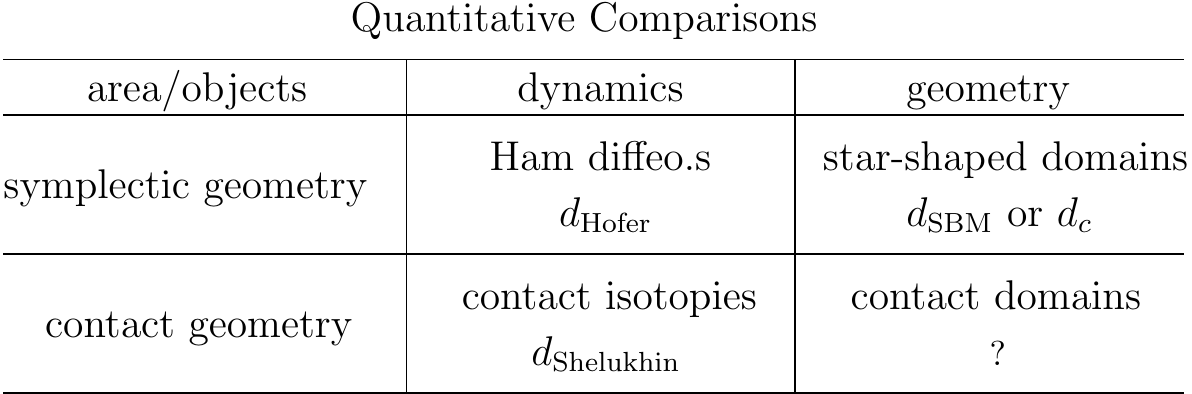}
\end{center}

\subsection{Pseudo-metric spaces from partial orders} In this paper, both the comparison between contactomorphisms, more precisely, certain equivalence classes of contact isotopies, and the comparison between contact domains will be heavily influenced by the concept, relative growth rate, which was first defined and investigated in \cite{EP00} (and also in a recent work \cite{CEP17} on certain non-compact contact manifolds). The relative growth rate can be defined in a broad way. In general, let $G$ be a semi-group endowed with a bi-invariant partial order $\geq$ \footnote{A partial order $\geq$ is bi-invariant if, for any $a,b \in G$, we have $a \geq b$ if and only if $ac \geq bc$ and $ca \geq cb$ for any $c \in G$.}.  We call an element $a \in G\backslash \{1\}$ {\it a dominant} if, for any $b \in G$, there exists some $k \in \N$ such that $a^k \geq b$. The set of all the dominants of $G$ is denoted by $G_+$. For $a \in G_+$ and $b \in G$, define  
\begin{equation} \label{dfn-rgr}
\rho^+_{\geq }(a,b) := \lim_{l \to \infty} \frac{\inf\{k \in \Z\,|\, a^k \geq b^l\}}{l} =  \inf\left\{ \frac{k}{l} \in \Q\, \bigg| \, (k \in \Z, l \in \N) \, (a^k \geq b^l)\right\}.
\end{equation}
For the second equality in (\ref{dfn-rgr}), see the proof of Theorem 2.4 in \cite{BS10}. If $b \in G_+$ as well, then similarly one can define $\rho^-_{\geq}(a,b) := \rho^+_{\geq}(b,a)$. The {\it relative growth rate between $a$ and $b$} is defined by $\gamma_{\geq}(a,b) := \max\{|\rho^+_{\geq}(a,b)|, |\rho^-_{\geq}(a,b)|\}$ for $a,b \in G_+$. In Section \ref{sec-app}, Appendix, we show that the relative growth rate $\gamma_{\geq}$  (in fact $\rho_{\geq}^+$) over $G_+$ defined in (\ref{dfn-rgr}) admits a further simplification which can be defined via only prime numbers. Moreover, if both $a, b \in G_+$, then both $\rho^+_{\geq} (a,b)$ and $\rho^-_{\geq}(a,b)$ are non-zero. This is due to the cute inequality (1.1.A) in \cite{EP00}, that is, $\rho^+_{\geq} (a,b) \rho^-_{\geq}(a,b) \geq 1$. One can use the relative growth rate to construct a pseudo-metric space. Explicitly, define $d_{\geq}: G_+ \times G_+ \to [0, \infty)$ by 
\begin{equation} \label{dfn-pm}
d_{\geq}(a, b) = \ln \gamma_{\geq}(a,b)
\end{equation}
for $a, b \in G_+$, and it is easy to check that $(G_+, d_{\geq})$ is a pseudo-metric space. The famous (non)-orderability phenomenon (\cite{EP00}, \cite{EKP06}) of contact manifolds is defined (or discovered) by a relation $\geq_o$ applied to $G := \widetilde{\rm Cont}_0(M, \xi)$, the universal covering of the identity component of the contactomorphism group of a contact manifold $(M, \xi)$, via a comparison between contact Hamiltonian functions. For details and its related background, see subsection \ref{subsec-pf-g1-prop}. 

For our own purpose, this paper will mainly consider a slightly different semi-group with a relation different from $\geq_o$. Let $(M, \xi)$ be a contact manifold with a co-oriented contact structure. Denote by ${\rm PCont}_0(M, \xi)$ the set of all the path (parametrized by $[0,1]$) inside contactomorphism group of $(M, \xi)$ starting from $\mathds{1}$, and by ${\rm PCont}_+(M, \xi)$ the set of all positive paths in ${\rm PCont}_0(M, \xi)$. Denote by $\overrightarrow{G} = {\rm PCont}_+(M, \xi)/{\sim}$, where $\phi \sim \psi$ if and only if $\phi_1 = \psi_1$ and they are homotopic with fixed endpoints through {\it positive} paths in ${\rm PCont}_+(M, \xi)$. For any two elements $[\phi]$ and $[\psi]$ in $\overrightarrow{G}$, assume that $[\phi] = [\{\phi_t\}_{t \in [0,1]}]$ and $[\psi] = [\{\psi_t\}_{t \in [0,1]}]$. Define 
\[ [\phi] \cdot [\psi] := [ \{\phi_t \psi_t\}_{t \in [0,1]}]. \]
It is easy to check that the operator ``$\cdot$'' defines a semi-group structure on $\overrightarrow{G}$. For brevity, we will omit the notation ``$\cdot$'', i.e., $[\phi][\psi]: = [\phi] \cdot [\psi]$. Moreover, there exists a natural projection $\pi: \overrightarrow{G} \to {\widetilde {\rm Cont}}_0(M, \xi)$. By definition, $(M, \xi)$ is orderable or non-orderable if and only if $\mathds{1} \notin \pi(\overrightarrow{G})$ or $\mathds{1} \in \pi(\overrightarrow{G})$, respectively, where $\mathds{1}$ is the identity class in $\widetilde{{\rm Cont}}_0(M, \xi)$ represented by a contractible loop. Finally, one can check that $\pi([\phi][\psi]) = \pi([\phi]) \pi([\psi])$, that is, $\pi$ is a semi-group morphism. Inspired by subsection 6.3 in \cite{EKP06}, for $[\phi], [\psi] \in \overrightarrow G$, define the following relation, 
\begin{equation} \label{order-2}
 [\phi] \geq_+ [\psi] \,\,\,\,\mbox{if and only if} \,\,\,\, [\phi] = [\psi] \,\,\mbox{or} \,\, [\phi] = [\psi][\theta] \,\,\mbox{for some $[\theta] \in \overrightarrow{G}$.}
\end{equation}
This relation $\geq_+$ applies more broadly than $\geq_o$ in the sense that $\geq_+$ can also apply, to some extent, to non-orderable contact manifolds (see (iv) in Proposition 6.10 in \cite{EKP06}). Let us focus on the orderable contact manifolds, and here is our first main result. 

\begin{theorem} \label{G1-prop}
Let $(M, \xi)$ be an orderable contact manifold. Then the relation $\geq_+$ in (\ref{order-2}) defines a bi-invariant partial order on $\overrightarrow{G}$. Moreover, the projection $\pi: \overrightarrow{G} \to {\widetilde {\rm Cont}}_0(M, \xi)$ is monotone, i.e., for any $[\phi], [\psi] \in \overrightarrow{G}$ such that $[\phi] \geq_+ [\psi]$, we have $\pi([\phi]) \geq_o \pi([\psi])$, where $\geq_o$ is the partial order on ${\widetilde {\rm Cont}}_0(M, \xi)$ defined in (\ref{po-ep}).
\end{theorem}

The proof of Theorem \ref{G1-prop} is presented in subsection \ref{subsec-pf-g1-prop}. This theorem immediately implies the following result. 

\begin{cor} \label{cor-po+}
Let $(M, \xi)$ be an orderable contact manifold. Then $(\overrightarrow{G}, d_{\geq _+})$ is a pseudo-metric space, where $d_{\geq_+}$ is defined as in (\ref{dfn-pm}) with respect to the partial order $\geq_+$, and the natural projection $\pi: \overrightarrow{G} \to \widetilde{\rm Cont}_0(M, \xi)$ is a monotone decreasing embedding $\pi: (\overrightarrow{G}, d_{\geq _+}) \to (G_+, d_{\geq_o})$, i.e., for any $[\phi], [\psi] \in \overrightarrow{G}$, we have $d_{\geq_+}([\phi], [\psi]) \geq d_{\geq_o}(\pi([\phi]), \pi([\psi]))$. 
\end{cor}

Therefore, the non-triviality (if it exists) of the pseudo-metric space $(G_+, d_{\geq_o})$ implies the non-triviality of the pseudo-metric space $(\overrightarrow{G}, d_{\geq _+})$ \footnote{We call a pseudo-metric space $(X, d)$ {\it  non-trivial} if there exist two elements $x,y \in X$ such that $d(x,y) >0$; otherwise $(X,d)$ is called a {\it trivial} pseudo-metric space.}. Similarly, a large-scale geometric property (if it exists) of the pseudo-metric space $(G_+, d_{\geq_o})$ implies a large-scale geometric property of the pseudo-metric space $(\overrightarrow{G}, d_{\geq _+})$. For non-orderable contact manifold, there does not exist such a conclusion as in Theorem \ref{G1-prop} (thus no such conclusion as in Corollary \ref{cor-po+}). The paragraph right after the proof of Theorem \ref{G1-prop} in subsection \ref{subsec-pf-g1-prop} provides a detailed explanation.  

\begin{remark} The necessity of working on ${\rm PCont}_+(M, \xi)$ or its quotient $\overrightarrow{G}$ will be clear when we reach subsection \ref{dyn-geo}. Roughly speaking, each positive contact isotopy corresponds to a contact version of fiberwise star-shaped domain of a contact manifold. The requirement of a homotopy through (only) positive contact isotopies guarantees that the induced deformation of the corresponding domains takes place within the class of fiberwise star-shaped domains. On the other hand, we emphasize that the partial order $\geq_o$ defined on ${\widetilde{\rm Cont}}_0(M, \xi)$ is not applicable to $\overrightarrow{G}$. It seems impossible to show that $\geq_o$ is anti-symmetric and transitive on $\overrightarrow{G}$, partially due to the fact that no inverse exists in $\overrightarrow{G}$. 
\end{remark}

\subsection{Contact Banach-Mazur distance} \label{section-cbm-intro} The relative growth rate also inspires the definition, contact Banach-Mazur distance $d_{\rm CBM}$, between contact domains. The comparison between domains in the contact geometry set-ups involves new ideas. Compared with the symplectic geometry set-ups, let us emphasize the following essential difference appearing in the contact geometry set-ups. Recall that one of the key ingredients in the definition of the (coarse) symplectic Banach-Mazur distance $d_c$ (see the definition (\ref{dfn-csbm})) in the symplectic geometry set-up is the ``rescaling'' of a star-shaped domain $U$ by a constant $C>0$, denoted by $CU$. This is achieved by flowing along the Liouville vector field, a defining element of a Liouville manifold. However, in the contact geometry set-ups, where no such canonical vector field exists, how to define this rescaling becomes unclear and ambiguous. To initiate our discussion, we will consider the following two contact geometry set-ups in this paper: one is the contactization of a Liouville manifold $W$, i.e., $W \times S^1$; the other is a closed Liouville-fillable contact manifold $M$, i.e., it can be viewed as the boundary of a Liouville domain. In the formal case, the objects in the discussion are fiberwise star-shaped domains $U \subset W \times S^1$ (its definition is given in the beginning of subsection \ref{subsec-pf-ccbm}); in the later case, the objects in the discussion are those domains constructed via contact 1-forms $\alpha$ on $M$, denoted by $W^{\alpha}$ (see the definition (\ref{geo-alpha})). We will discuss these two cases separately.

\subsubsection{$d_{\rm CBM}$ between domains} For the quantitative comparison between fiberwise star-shaped domains of $W \times S^1$, we will use the following key definition, which can be viewed as a rescaling in this contact geometry set-up.

\begin{dfn} \label{dfn-cont-rescale} For a fiberwise star-shaped domain $U$ of the contact manifold $W \times S^1$, and any $k \in \N$, define $U/k: = \tau^{-1}_k(U)$, where $\tau_k$ is defined as follows,   
\begin{equation} \label{cover}
\tau_k: W \times S^1 \to W \times S^1 \,\,\,\,\mbox{by} \,\,\,\, \tau_k(p, t) = \left(\phi_L^{\ln k}(p), k t\right),
\end{equation}
for any point $(p,t) \in W \times S^1$.
 \end{dfn}

Roughly speaking, $U/k$ not only shrinks the $W$-component of $U$, but also repeats this $W$-component for $k$ many times along the $S^1$ factor. In particular, if $U$ is a split domain, i.e., $U = \check{U} \times S^1$ for some star-shaped domain $\check{U} \subset W$, then $U/k = \tau^{-1}_k(U) = \phi_{L}^{-\ln k}(\check{U}) \times S^1$. Here is a more concrete example. 

\begin{ex} \label{ex-rescale} Let $W = \R^{2n}$ and $U = B^{2n}(R) \times S^1$. Endow $W$ with the standard Liouville vector field $L = \frac{1}{2}(x \partial_x + y \partial_y)$ and $\omega = \sum_{i=1}^n dx_i \wedge dy_i = d\lambda_{\rm std}$ where $\lambda_{\rm std} = \frac{1}{2} \sum_{i=1}^n x_i dy_i -y_i dx_i$. Then it is easy to check that for any $p = (x_1, y_1, ..., x_n, y_n) \in W$, 
\[ \phi_L^{-\ln k}(p) = \left(\frac{x_1}{\sqrt{k}},\frac{y_1}{\sqrt{k}}, ..., \frac{x_n}{\sqrt{k}}, \frac{y_n}{\sqrt{k}}\right) = \frac{p}{\sqrt{k}}. \]
Meanwhile, recall that $B^{2n}(R) = \{p \in \R^{2n} \,|\, \pi |p|^2 < R\}$. Hence, by Definition \ref{dfn-cont-rescale}, $(B^{2n}(R) \times S^1)/k \simeq B^{2n}(R/k) \times S^1$, where the $S^1$ factor on the right-hand side should be regarded as the $k$-th cover of the $S^1$ factor on the left-hand side. 
\end{ex}

\begin{remark} \label{rmk-cont} One can also define the covering map $\tau_k$ to be $\tau_k(p, t) = (\phi_L^{-\ln k}(p), k t)$. Then, instead of shrinking the $W$-component of a given fiberwise star-shaped domain $U$, $\tau^{-1}_k(U)$ will amplify the $W$-component of $U$. The reason why we chose our covering map as in (\ref{cover}) is that $\tau_k$ preserves the contact 1-form $\alpha_{\rm std} = \lambda_{\rm std} + dt$ up to a factor $k$ (hence, $\tau_k$ preserves the contact structure of $W \times S^1$).  \end{remark}

\begin{remark} We emphasize that, for $k<l$, it is not necessarily true that $U/l \subset U/k$ (but this holds trivially if $U$ is a split domain).  For instance, if $U$ is a 2-disk bundle over $S^1 = \R/\Z$ such that for $t \in S^1$ the corresponding fiber $U|_{t}$ is a 2-disk of the radius $r_t = \cos(2\pi t) + 2$. Then $U/2$ is also a 2-disk bundle over $S^1$ such that for $t \in S^1$ the corresponding fiber $(U/2)|_t$ is a 2-disk of the radius $r'_t = \frac{\cos(4\pi t)}{2} + 1$. Observe that when $t = \frac{1}{2}$, $r_{\frac{1}{2}} = 1$ while $r'_{\frac{1}{2}} = \frac{3}{2}$. Since $r'_{\frac{1}{2}} > r_{\frac{1}{2}}$, $U/2$ is not contained in $U$. 
\end{remark}

Here is a contact geometric analogue of the symplectic comparison $d_c$ defined in (\ref{dfn-csbm}). 

\begin{dfn} \label{dfn-ccbm} Let $U, V$ be two fiberwise star-shaped domains of the contact manifold $W \times S^1$ where $W$ is a Liouville manifold. Define the {\it contact Banach-Mazur distance} $d_{\rm CBM}(U,V)$ between $U$ and $V$ as follows. First, define
\begin{equation}\label{rho-c} 
\rho_c(U, V)= \inf\left\{\frac{k}{l} \in \Q_+\,\bigg|\, \begin{array}{l} \exists\, \mbox{a compactly supported contact isotopy } \\ \mbox{$\phi = \{\phi_t\}_{t \in [0,1]}$ on $W \times S^1$ s.t. $\phi_1(U/k) \subset V/l$} \end{array}\right\},
\end{equation}
where rescaling $U/k$ and $V/l$ are defined in Definition \ref{dfn-cont-rescale}. Then define
\begin{equation*}
\gamma_{\rm CBM}(U, V) = \max\{\rho_c(U, V), \rho_c(V,U)\}. 
\end{equation*}
Finally, $d_{\rm CBM}(U,V) = \ln \gamma_{\rm CBM}(U,V)$ (if taking $\ln$ is applicable).
\end{dfn}

Different from the coarse symplectic Banach-Mazur distance $d_c$, the contact Banach-Mazur distance $d_{\rm CBM}(U,V)$ sometimes is not well-defined. The subsection \ref{subsec-pf-ccbm} provides a detailed discussion on this issue, which depends on (non)-squeezable domains (see Definition \ref{dfn-sq}). Denote by $\mathcal N_{W \times S^1}$ the set of all the non-squeezable fiberwise star-shaped domains of $W\times S^1$. The following theorem is our second main result. 

\begin{theorem} \label{lemma-ccbm-property} 
Suppose that $W \times S^1$ satisfies the condition that $\mathcal N_{W \times S^1} \neq \emptyset$. Then $d_{\rm CBM}$ defines a non-trivial pseudo-metric on $\mathcal N_{W \times S^1}$. Moreover, for any compactly supported contactomorphisms $\phi, \psi \in {\rm Cont}_0(W \times S^1)$ and any fiberwise star-shaped domains $U,V$ in $\mathcal N_{W \times S^1}$, we have $d_{\rm CBM}(\phi(U), \psi(V)) = d_{\rm CBM}(U, V)$.
\end{theorem}

To justify the analogue between $d_{\rm CBM}$ and $d_c$, let us make the following observation. Let $U,V$ be two star-shaped domains of Liouville manifold $(W, \omega, L)$ and $\lambda = \iota_L \omega$. Any compactly supported Hamiltonian isotopy $\phi = \{\phi_t\}_{t \in [0,1]}$ on $W$ admits a lift as a (strict) contact isotopy $\tilde{\phi} = \{\tilde{\phi}_t\}_{t \in [0,1]}$ defined on $W \times S^1$ with respect to the contact 1-form $\lambda + dt$. Explicitly, for any $(x,t) \in W \times S^1$, 
\begin{equation} \label{lift-cont}
\tilde{\phi}_t (x,t) = (\phi_t(x), (t+ f_t(x)) \,\,\mbox{mod}\,\,1) 
\end{equation}
where $f_t$ satisfies $\phi_t^*{\lambda} = \lambda - df_t$. Then we have the following property. 

\begin{prop} \label{prop-cbm-c} Let $U,V$ be star-shaped domains of Liouville manifold $(W, \omega, L)$. Then we have the following relation, 
\[ d_{\rm CBM}(U \times S^1, V \times S^1) \leq d_{c}(U,V).\]
In particular, $d_{\rm CBM}(U \times S^1, V \times S^1) \leq d_{\rm SBM}(U,V)$
 \end{prop}

\begin{remark} \label{rmk-dc} It is now well-understood that the study of $d_c$ between star-shaped domains of a Liouville manifold is more difficult than the study of $d_{\rm SBM}$ since (filtered) symplectic homology does not apply (\cite{GU19}, \cite{HZ19}). Fortunately, the relation in Proposition \ref{prop-cbm-c} provides a new approach to study $d_c$ via $d_{\rm CBM}$. However, due to Proposition \ref{prop-gamma-0}, in order to obtain a meaningful lower bound of $d_c(U,V)$, we need to consider those Liouville manifolds $W$ with certain constraints. In particular, $d_{\rm CBM}$ is not applicable to the study of $d_c$ between star-shaped domains of $\R^{2n}$ . \end{remark}

The proofs of both Theorem \ref{lemma-ccbm-property} and Proposition \ref{prop-cbm-c} will be presented in subsection \ref{subsec-pf-ccbm}. 

\subsubsection{$d_{\rm CBM}$ between forms}\footnote{This is based on a discussion with Matthias Meiwes.}  \label{subsec-cbm-forms} This subsection discusses another version of the contact Banach-Mazur distance. Different from the set-up in the previous subsection, this comparison takes place between contact 1-forms. Explicitly, let $(M, \xi)$ be a contact manifold, and assume that $\xi$ is co-oriented. This contact structure $\xi$ corresponds to a family of contact 1-forms on $M$. Fix a contact 1-form $\alpha_0$ on $M$, then this family can be written as an orbit space 
\[ O_{\xi}(\alpha_0) = \{e^f \alpha_0 \,| \, f \in C^{\infty}(M, \R) \} : = C^{\infty}(M, \R) \cdot \{\alpha_0\}, \]
where the action of $C^{\infty}(M; \R)$ on the base point $\{\alpha_0\}$ is a multiplication by $e^f$ for any $f \in C^{\infty}(M, \R)$. Recall that any contactomorphism $\phi \in {\rm Cont}(M, \xi)$ satisfies $\phi_* \xi = \xi$. Equivalently, for any $\alpha \in O_{\xi}(\alpha_0)$, $\phi^*\alpha = e^{g_{\phi, \alpha}} \alpha$ where $g_{\phi, \alpha}$ is the conformal factor of $\phi$ with respect to $\alpha$. Define a partial order $\preceq$ on elements in $O_{\xi}(\alpha_0)$ as follows: $\alpha_1 \preceq \alpha_2$ if and only if the corresponding functions $f_1$ and $f_2$ such that $\alpha_1 = e^{f_1} \alpha_0$ and $\alpha_2 = e^{f_2} \alpha_0$ satisfy the relation $f_1 \leq f_2$, that is, $f_1(p) \leq f_2(p)$ for any $p \in M$. For any constant $C \in \R_+$, denote by $C \alpha$ the rescaling $e^{f+ \ln C} \alpha_0$ if $\alpha = e^f \alpha_0$. Then we give the following definition. 

\begin{dfn} \label{dfn-cbm-forms}
For any $\alpha_1, \alpha_2 \in O_{\xi}(\alpha_0)$, define the contact Banach-Mazur distance as  
\[ d_{{\rm CBM}}(\alpha_1, \alpha_2) := \inf \left\{\ln C \geq 0 \,\bigg| \, \exists \phi \in {\rm Cont}_0(M, \xi) \,\,\mbox{s.t.} \,\, \frac{1}{C} \alpha_1 \preceq \phi^*\alpha_2 \preceq C \alpha_1\right\}. \]
\end{dfn}

\begin{remark} Since the input of $d_{\rm CBM}$ speaks itself in an obvious way that which version of contact Banach-Mazur distance is in the discussion, we adopt the same notation as the one in Definition \ref{dfn-ccbm}. \end{remark}

It is easy to check that the space $(O_{\xi}(\alpha_0), d_{\rm CBM})$ is a pseudo-metric space (see Proposition \ref{prop-cbm}) and $d_{\rm CBM}$ is ${\rm Cont}_0(M, \xi)$-invariant, i.e., $d_{\rm CBM}(\phi^*\alpha_1, \psi^*\alpha_2) = d_{\rm CBM}(\alpha_1, \alpha_2)$ for any $\phi, \psi \in {\rm Cont}_0(M, \xi)$. Suppose that $(M, \xi = \ker\alpha_0)$ is closed and Liouville-fillable, that is, there exists a Liouville domain $(W, \omega, L)$ such that $\partial W = M$, $(\iota_L\omega)|_M = \alpha_0$ and the flow of $L$ is complete for $t<0$. In this case, there is a geometric perspective to view any element $\alpha \in O_{\xi}(\alpha_0)$ (cf.~Section 3.2 in \cite{AM19}). Denote by $(\hat{W}, \omega)$ the completion of $(W, \omega, L)$. With respect to $\alpha_0$, $\hat{W}$ admits a canonical decomposition $\hat{W} = SM \sqcup {\rm Core}(W)$ where $SM$ is the symplectization of $(M, \xi = \ker\alpha_0)$ and it can be identified with $\R_+ \times M$, in coordinate $(u, x)$. In this coordinate $W = \{u\leq 1\}$, $M = \{u = 1\}$ and ${\rm Core}(W) = \{u=0\}$. 
Let $\alpha \in O_{\xi}(\alpha_0)$ and assume that $\alpha = e^f \alpha_0$ for some function $f: M \to \R$. Define 
\begin{equation} \label{geo-alpha}
W^{\alpha} := \left\{(u, x) \in \hat{W} \,| \, u < e^{f(x)} \right\}.
\end{equation}
Observe that $W^{\alpha}$ is also a Liouville domain of $(\hat{W}, \omega, L)$. Then we have another main result as follows. 

\begin{theorem} \label{thm-dc-lb} For any $\alpha_1, \alpha_2 \in O_{\xi}(\alpha_0)$, $d_{\rm SBM}(W^{\alpha_1}, W^{\alpha_2}) \leq d_{\rm CBM}(\alpha_1, \alpha_2)$. In particular, $d_{c}(W^{\alpha_1}, W^{\alpha_2}) \leq d_{\rm CBM}(\alpha_1, \alpha_2)$. 
\end{theorem}

The proof of Theorem \ref{thm-dc-lb} will be provided in subsection \ref{subsec-pf-dc-lb}. We point out that $d_{\rm CBM}$ could vanish identically, without any constraint on $(M, \xi)$. Even though one considers the rescaling $C\alpha_0$ (with the expectation that $d_{\rm CBM}(C \alpha_0, \alpha_0) = \ln C$), the inequality $d_{\rm CBM}(C \alpha_0, \alpha_0)>0$ does not always hold. One easy example is $\R^3$ with $\alpha_0 = \frac{1}{2} (xdy - ydx) + dz$. The rescaling of $\alpha_0$ can be realized by a contactomorphism, which implies that $d_{\rm CBM}(C\alpha_0, \alpha_0) = 0$. This is in contrast with the rescaling property $d_{c}(CU,U) = \ln C$ in the symplectic case. For some cases, this issue can be fixed by Theorem \ref{thm-dc-lb} as demonstrated by the following example.

\begin{ex} \label{ex-nonzero-cbm} Suppose that $(M, \xi = \ker \alpha_0)$ is compact and Liouville-fillable, and $W^{\alpha_0}$ is its filling. Since by definition $d_{\rm SBM}(CU, U) = \ln C$, we know 
\begin{equation} \label{non-trivial-cbm}
\ln C = d_{\rm SBM}(CW^{\alpha_0}, W^{\alpha_0}) = d_{\rm SBM}(W^{C\alpha_0}, W^{\alpha_0}) \leq d_{\rm CBM}(C\alpha_0, \alpha_0) \leq \ln C,
\end{equation}
which implies that $d_{\rm CBM}(C\alpha_0, \alpha_0) = \ln C$. This relation also holds for any $\alpha \in O_{\xi}(\alpha_0)$. In particular, $d_{\rm CBM}$ is not identically zero. More generally, without the hypothesis of being Liouville-fillable (so we only consider the symplectization of $M$), by the compactness of $M$ we know ${\rm Vol}(\{s<1\})< \infty$ where the volume {\rm Vol} is taken inside $(SM, \omega)$. Then, similarly to the argument in (\ref{non-trivial-cbm}), one can also show that $d_{\rm CBM}(C\alpha_0, \alpha_0) = \ln C$ simply by the volume-preserving property of symplectomorphisms. \end{ex}

\subsection{Results on large-scale geometry} Similarly to the main results in \cite{SZ18} and \cite{Ush18}, a driving force of introducing and studying the distances between objects in the contact geometry set-ups is to investigate large-scale geometric properties of the spaces of such objects with respect to the corresponding distances. The general scheme of obtaining such properties is to establish certain stability results with respect to some algebraic obstructions, which can provide large gaps and also serve as the lower bound of the distances in our discussion. Here are our main results in this direction. 

\medskip

For $d_{\rm CBM}$ between domains, we will use the contact shape invariant (\cite{Eli91}) of split domains $U \times S^1$ of $W \times S^1$ to provide a meaningful lower bound of $d_{\rm CBM}$. For two subsets $A, B$ of $\R^n$ containing $0 \in \R^n$, define $\delta(A,B) := \inf \{C>1 \,| \, \frac{1}{C} \cdot A \subset B \subset C \cdot A\}$. Then we have the following stability result.

\begin{theorem} \label{thm-stability} Let $W= T^*\T^n$, and $U = \T^n \times A_U$ and $V = \T^n \times A_V$, where $A_U, A_V$ are subsets of $\R^n$ containing $0 \in \R^n$. Then 
\[  \delta(A_U, A_V) \leq d_{\rm CBM}(U \times S^1, V \times S^1).\]
\end{theorem}

A brief background of contact shape invariant, as well as its key properties in the set-up in this paper, will be recalled and demonstrated in subsection \ref{ssec-csi}. Denote by $\mathcal S(T^*\T^n)$ the set of all the star-shaped domains of $T^*\T^n$. Then Lemma \ref{prop-cbm-c} and Theorem \ref{thm-stability} imply the following large-scale geometric property of the pseudo-metric space $(\mathcal S(T^*\T^n), d_c)$. (cf.~Remark \ref{rmk-dc}). 

\begin{cor} \label{cor-large-scale-torus} If $n \geq 2$, then for any $k \in \N$, there exists an embedding $\Psi_k: (\R^k, |\cdot|_{\infty}) \to (\mathcal S(T^*\T^n), d_c)$ and a constant $C>1$ such that for any $v, w \in \R^k$,
\[ \frac{1}{2}|v- w|_{\infty} \leq d_c(\Psi_k(v), \Psi_k(w)) \leq C |v - w|_{\infty}.\]
In other words, $\Psi_k$ is a quasi-isometric embedding for any $k \in \N$.  
\end{cor}

\begin{remark} When $n =1$, the real dimension of $T^*S^1$ is two, which implies that the only symplectic invariant of the star-shaped domains of $T^*S^1$ is just the area. Therefore, the conclusion in Corollary \ref{cor-large-scale-torus} does not hold in this case, and then the hypothesis $n \geq 2$ in Corollary \ref{cor-large-scale-torus} is necessary. \end{remark}

\begin{remark} \label{rmk-hz} Another proof of Corollary \ref{cor-large-scale-torus}, using the symplectic shape invariant (cf.~(\ref{dfn-ssi})), will appear in a forthcoming work \cite{HZ19}, where the coarse symplectic Banach-Mazur distance $d_c$ considered in \cite{HZ19} is slightly more general than the one defined in (\ref{dfn-csbm}). The way that proves Corollary \ref{cor-large-scale-torus} in this paper via Theorem \ref{thm-stability} can be viewed as a refinement of the argument in \cite{HZ19}. As a matter of fact, \cite{HZ19} proves the same conclusions as in Corollary \ref{cor-large-scale-torus} for $T^*(\T^n \times N)$, where $n \geq 2$ and $N$ is any closed manifold. \end{remark}

The proofs of both Theorem \ref{thm-stability} and Corollary \ref{cor-large-scale-torus} will be presented in subsection \ref{ssect-proof-csi}. For $d_{\rm CBM}$ between forms, based on the main results in \cite{SZ18} and \cite{Ush18} where the algebraic obstructions were obtained from persistence modules theory (\cite{PRSZ19}), very different from the application of the contact shape invariant above, we obtain the following result. 

\begin{theorem} \label{cont-large-scale} Let $(M, \xi = \ker \alpha_0)$ be either $(S^{2n-1}, \xi_{\rm std} = \ker \alpha_{\rm std})$ where $n \geq 2$ or $(S^*_g \Sigma_{\geq 1}, \xi_{\rm can} = \ker \alpha_{\rm can})$ where $S^*_g \Sigma_{\geq 1}$ denotes the unit co-sphere bundle over an oriented surface $\Sigma_{\geq 1}$ of genus at least 1. Then for any $N \in \N$, there exists a quasi-isometric embedding from $(\R^N, |\cdot|_{\infty})$ into $(O_{\xi}(\alpha_0), d_{\rm CBM})$ where $\alpha_0 = \alpha_{\rm std}$ or $\alpha_0 = \alpha_{\rm can}$, respectively.\end{theorem}

\begin{proof} Since both $(S^{2n-1}, \xi = \ker \alpha_{\rm std})$ and $(S^*_g \Sigma_{\geq 1}, \xi = \ker \alpha_{\rm can})$ are compact and Liouville-fillable, the desired conclusion directly comes from Theorem 1.5 in \cite{Ush18} and Theorem 1.11 in \cite{SZ18} for two different situations in the hypothesis respectively, together with the stability result Theorem \ref{thm-dc-lb} above. \end{proof}

\begin{remark} On the one hand, the conclusion in Theorem \ref{cont-large-scale} seems not surprising since as a set the space $O_{\xi}(\alpha_0)$ is identified with the space $C^{\infty}(M, \R)$, which tends to be large. On the other hand, the condition in the definition of $d_{\rm CBM}$ in Definition \ref{dfn-cbm-forms} can be rewritten in terms of functions. Explicitly, assume that $\alpha_1 = e^{f_1} \alpha_0$ and $\alpha_2 = e^{f_2} \alpha_0$, then it is easy to check that the relation $(1/C) \alpha_1 \preceq \phi^*\alpha_2 \preceq C\alpha_1$ is equivalent to 
\begin{equation} \label{cbm-fcn}
f_1 - \ln C \leq f_2 \circ \phi + g_{\phi, \alpha_0} \leq f_1 + \ln C,
\end{equation}
where $g_{\phi, \alpha_0}$ is the conformal factor of $\phi$ with respect to $\alpha_0$. Due to the existence of $g_{\phi, \alpha_0}$ in the comparison (\ref{cbm-fcn}) which depends on $\phi$, seeking for the optimal $C$ in (\ref{cbm-fcn}), i.e., the value of $d_{\rm CBM}$, becomes non-trivial at all. \end{remark} 

\subsection{Special properties of contact isotopies} There are two more special parts in this paper, which take advantages of two particular properties of contact isotopies. (1) Fix a contact 1-form, then a contact isotopy admits a unique contact Hamiltonian function. Based on the construction in \cite{EP00} and \cite{EKP06}, this contact Hamiltonian function can be used to transfer this contact isotopy to a fiberwise star-shaped domain $U \subset W \times S^1$, which lies exactly in our interest from the discussion above. (2) A contact isotopy of a contact manifold $M$ lifts to a 1-homogenous symplectic Hamiltonian isotopy of the symplectization of $M$. When $M = S_g^*X$, the unit co-sphere bundle of a manifold $X$, the groundbreaking work in \cite{GKS12} associates to such a 1-homogenous symplectic Hamiltonian isotopy an element in the derived category of sheaves of modules over $X \times X \times I$, called the sheaf quantization. The following two subsections will discuss several related results in these two directions. 

\subsubsection{From dynamics to geometry} \label{dyn-geo} Let $(M, \xi)$ be a closed contact manifold. Recall that $\overrightarrow{G} = {\rm PCont}_+(M, \xi)/{\sim}$ where $\phi \sim \psi$ if and only if $\phi_1 = \psi_1$ and they are homotopic with fixed endpoints through positive paths in ${\rm PCont}_+(M, \xi)$. Consider an element $[\phi] \in \overrightarrow{G}$. Due to Lemma 3.1.A in \cite{EP00}, there exists a representative of the class $[\phi]$ which is 1-periodic and denote by $\phi = \{\phi_t\}_{t \in S^1}$ where $S^1 = \R/\Z$. Fix a contact 1-form $\alpha$, and denote by $h(t,x): S^1 \times M \to \R$ the contact Hamiltonian function of $\phi$. From the discussion above, it lifts to a 1-homogenous (symplectic) Hamiltonian function $H(s,t,x) = s\cdot h(t,x)$ on symplectization $SM \times S^1$. Assume that $(M, \xi)$ is Liouville-fillable and denote by $W$ its filling, and by $\hat{W}$ the completion of $W$. Suggested by \cite{EKP06}, consider the following subset of $\hat{W} \times S^1$, 
\begin{equation} \label{geo-ss}
U(\phi) = \{(s, t, x) \in SM \times S^1 \,| \, H(s,t,x) <1 \} \cup ({\rm Core}(\hat{W}) \times S^1).
\end{equation}
It is easy to check that $U(\phi)$ is a star-shaped domain of Liouville manifold $\hat{W} \times S^1$. More explicitly, since $h(t,x)>0$ for any $(t,x) \in S^1 \times M$, due to the compactness of $M$, there exist a global minimum and a global maximum of $h(t,x)$, denote by $m_-$ and $m_+$, respectively. Both $m_-$ and $m_+$ are positive. The condition in (\ref{geo-ss}) $H(s,t,x)<1$ is equivalent to $h(t,x) < \frac{1}{s}$. Therefore, if $s \geq \frac{1}{m_-}$, then $U(\phi)|_{s} = \emptyset$, and if $s \leq \frac{1}{m_+}$, the $U(\phi)|_{s} =  S^1 \times M$. One standard example of this construction is that $M = S^{2n-1}$, the standard sphere of $\R^{2n}$ and $\phi = \{\phi_t\}_{t \in [0,1]} = \{e^{2\pi i t} z\}_{t\in [0,1]}$ the 1-periodic Reeb flow on $S^{2n-1}$. The corresponding contact Hamiltonian function is $\pi|z|^2$, and then $U(\phi) = B^{2n}(1) \times S^1 \subset \R^{2n} \times S^1$. One can modify the speed of the Reeb flow to adjust the radius of the ball. One important observation is that an autonomous positive contact isotopy, i.e., the contact Hamiltonian function is independent of $t \in S^1$, corresponds to a split fiberwise star-shaped domain of $\hat{W} \times S^1$. 

Due to Lemma 1.21 in \cite{EKP06}, the construction in (\ref{geo-ss}) is in fact well-defined for the class $[\phi] \in \overrightarrow{G}$, up to an ambient contactomorphism in ${\rm Cont}_0({\hat W} \times S^1)$. We use the notation $U([\phi])$ to denote the equivalence class of the fiberwise star-shaped domains in ${\hat W} \times S^1$ up to contactomorphisms. Suppose that $(M, \xi)$ is orientable. Given two $[\phi], [\psi] \in \overrightarrow{G}$, on the one hand, due to (\ref{G1-prop}) and definition (\ref{dfn-rgr}), we can compare them via the relative growth rate $\gamma_{\geq_+}([\phi], [\psi])$ with respect to the bi-invariant partial order $\geq_+$ defined in (\ref{order-2}); on the other hand, by passing to the corresponding fiberwise star-shaped domains in ${\hat W} \times S^1$ defined in (\ref{geo-ss}), one can also compare them via the contact Banach-Mazur distance in Definition \ref{dfn-ccbm}, i.e., $d_{\rm CBM}(U([\phi]), U([\psi]))$. Due to the second conclusion of Proposition \ref{lemma-ccbm-property}, the value $d_{\rm CBM}(U([\phi]), U([\psi]))$ is well-defined. The following theorem shows that these two approaches are related, and its proof will be given in subsection \ref{subsec-pf-rgr-cbm}. 

\begin{theorem} \label{rgr-cbm}
Let $(M, \xi)$ be a compact and Liouville-fillable contact manifold, and $\hat{W}$ be the completion of its filling. Suppose that $\mathcal N_{\hat{W} \times S^1} \neq \emptyset$ and denote by $\overrightarrow{G}$ the semi-group of the equivalence classes of positive contact isotopies on $(M, \xi)$. Then for any $[\phi], [\psi] \in  \overrightarrow{G}$, we have 
\[ d_{\geq_+}([\phi], [\psi]) \geq d_{\rm CBM}(U([\phi]), U([\psi])),\]
where $[\phi], [\psi]$ are the equivalence classes of fiberwise star-shaped domains in $\mathcal N_{\hat W \times S^1}$ up to contactomorphisms. 
\end{theorem}

\begin{remark} Consider the following subset of $\overrightarrow{G}$, 
\[ \overrightarrow{A} : = \{[\phi] \in \overrightarrow{G} \,| \, \mbox{$\phi$ is an autonomous contact isotopy}\}. \]

(1) For any $[\phi] \in \overrightarrow{A}$, as mentioned above, $U([\phi])$ is a split domain up to a contactomorphism of ${\hat W} \times S^1$. Then the right hand side of the conclusion in Theorem \ref{rgr-cbm} reduces to $d_{\rm CBM}(U \times S^1, V \times S^1)$ for some star-shaped domains $U, V$ of $\hat{W}$. Then, for certain contact manifolds, e.g., $M = S_g^*(\T^n \times N)$ with $n \geq 2$ and any closed manifold $N$, the proof of Corollary \ref{cor-large-scale-torus} (also see Remark \ref{rmk-hz}) implies a large-scale geometric property of $(\overrightarrow{A}, d_{\geq_+})$, and then also of $(\overrightarrow{G}, d_{\geq_+})$. More explicitly, both $(\overrightarrow{A}, d_{\geq_+})$ and $(\overrightarrow{G}, d_{\geq_+})$ contain an image of a quasi-isometric embedding of an arbitrarily large dimensional Euclidean space. In the case where $M = S_g^*\T^n$ with $n \geq 2$, a large-scale geometric property of $(\overrightarrow{A}, d_{\geq_+})$ can also be obtained by Theorem 1.7.F in \cite{EP00} together with Corollary \ref{cor-po+} above. It seems to us that these two approaches do not coincide, essentially because the way that we use the contact shape invariant in this paper is different from \cite{EP00}. 

(2) The way that we obtain the large-scale geometric property of $(\overrightarrow{A}, d_{\geq_+})$ elaborated in (1) above can be regarded as an analogue of the proof of the main result, Theorem 1.1, in Usher's work \cite{Ush13}, where it proves that, when the symplectic manifold $(M, \omega)$ satisfies a certain condition, the group ${\rm Ham}(M, \omega)$ contains an image of a quasi-isometric embedding of an infinite-dimensional Euclidean space. From the construction of the quasi-isometric embedding in \cite{Ush13}, this large-scale geometric property in fact holds for the subset of all the autonomous Hamiltonian diffeomorphisms (i.e., generated by autonomous Hamiltonian functions). 

On the other hand, for certain symplectic manifolds $(M, \omega)$, Polterovich-Shelukhin's work \cite{PS16} (and also \cite{Zha19}) provides another method that obtains a large-scale geometric property of ${\rm Ham}(M, \omega)$, which goes beyond the scope of automorphism Hamiltonian diffeomorphisms. More explicitly, it finds a sequence of non-autonomous Hamiltonian diffeomorphisms whose Hofer distances $d_{\rm Hofer}$ from the set of autonomous Hamiltonian diffeomorphisms go to infinity. Therefore, an analogue and also interesting question in our contact set-up is whether there exists a sequence of non-split equivalence classes of fiberwise star-shaped domains in $\hat{W} \times S^1$ (i.e., not a split domain up to any contactomorphisms) such that their contact Banach-Mazur distances $d_{\rm CBM}$ from the set of all split equivalence classes can be large. To this end, one needs to understand how to use obstructions (for instance, the contact shape invariant applied in this paper) to effectively distinguish non-split and split equivalence classes. This certainly deserves some further development. 
\end{remark}

\subsubsection{Algebraic distance between sheaves} Let $\k$ be a fixed ground field, and $X$ be a closed manifold. Denote by $\D^b(\k_{X \times X \times I})$ the bounded derived category of sheaves of $\k$-modules over $X \times X \times I$ where $I$ is an interval of $\R$ containing $0$. The sheaf convolution denoted by $``\circ|_I"$ (see (1.13) in \cite{GKS12} or subsection \ref{subsec-pf-prop-sheaf-order} below) provides a well-defined bi-operator on $\D^b(\k_{X \times X \times I})$. Moreover, the derived category $\D^b(\k_{X \times X \times I})$ contains a subcategory denoted by $\D_{\rm adm}^b(\k_{X \times X \times I})$ consisting of those objects which, roughly speaking, admit its inverse under the sheaf convolution. Recall that the singular support of an element $\F \in \D^b(\k_{X \times X \times I})$, denoted by $SS(\F)$, is a conical subset of $T^*(X \times X \times I) = T^*X \times T^*X \times T^*I$. We will be mainly interested in the $T^*I$-component where $(r,\tau)$ denotes the coordinates. We use the notation $T^*_{\{\tau \leq 0\}} (X \times X \times I)$ to denote the subset of $T^*(X \times X \times I)$ where the co-vector $\tau$ satisfies the condition that $\tau \leq 0$.

\begin{dfn} \label{sheaf-order} For $\F, \G \in \D^b(\k_{X \times X \times I})$, define the relation $\geq_s$ between $\F$ and $\G$ by 
\[ \F \geq_s \G \,\,\,\,\mbox{if and only if} \,\,\,\, SS(\F^{-1} \circ|_I \G) \subset T^*_{\{\tau \leq 0\}} (X \times X \times I).\]
We call $\F$ a {\it dominant} if, for any $\G \in \D^b(\k_{X \times X \times I})$, there exists some $k \in \N$ (depending on $\G$) such that such that $\F^k \geq_s G$, where $\F^k: =\F \circ |_I \circ \cdots \circ|_I \circ \F$ for $k$ many $\F$. The set of all the dominants in $\D^b(\k_{X \times X \times I})$ is denoted by $\D_{\rm dom}^b(\k_{X \times X \times I})$.
\end{dfn}

For brevity, denote by $\D_{{\rm adm} \cap {\rm dom}}^b(\k_{X \times X \times I}) : = \D_{\rm adm}^b(\k_{X \times X \times I}) \cap \D_{\rm dom}^b(\k_{X \times X \times I})$. Recall that ${\rm PCont}_+(M, \xi)$ denotes the set of all the positive path of contactomorphisms of $(M, \xi)$ starting at $\mathds{1}$. Then we have the following main result, and its proof will be given in subsection \ref{subsec-pf-prop-sheaf-order}.

\begin{theorem} \label{prop-sheaf-order} The relation $\geq_s$ defined in Definition \ref{sheaf-order} is a partial order on the subcategory $\D_{\rm adm}^b(\k_{X \times X \times I})$. Moreover, if $M = S_g^*X$, the unit co-sphere bundle with the canonical contact structure $\xi$, then there exists a monotone decreasing embedding 
\begin{equation} \label{emb-sheaf}
\sigma: ({\rm PCont}_+(M, \xi), d_{\geq_o}) \to (\D_{{\rm adm} \cap {\rm dom}}^b(\k_{X \times X \times I}), d_{\geq s}),
\end{equation}
where $d_{\geq_o}$ and $d_{\geq_s}$ are the pseudo-metrics induced by the partial order $\geq_o$ and $\geq_s$ respectively via the relative growth rate as constructed in (\ref{dfn-pm}). 
\end{theorem}

Not surprisingly, the embedding in Theorem \ref{prop-sheaf-order} is given by the sheaf quantization from \cite{GKS12} or Example \ref{ex-sq} in subsection \ref{subsec-pf-prop-sheaf-order}, which associates to a contact isotopy $\phi$ a unique (up to isomorphism) element $\K_{\phi} \in \D_{\rm adm}^b(\k_{X \times X \times I})$, but the proof that the relation $\geq_s$ defined in Definition \ref{sheaf-order} is a partial order heavily depends on certain particular properties of singular supports (see Proposition \ref{prop-ss-1} and Proposition \ref{prop-ss-2}). 

\medskip

\noindent {\bf Discussion.} Let us end this subsection with the following discussion. Let $I = [0,1]$. By Proposition 4.3 in \cite{Zha20}, when $M = S_g^*X$ for some metric $g$ on $X$, there exists a well-defined embedding 
\begin{equation} \label{emb-sheaf-2}
[\sigma]: \widetilde{{\rm Cont}}_0(M, \xi) \to \D_{{\rm adm}}^b(\k_{X \times X \times I})/\sim,
\end{equation} 
which can be regarded as a homotopy version of the embedding in (\ref{emb-sheaf}), on the level of sets instead of the pseudo-metric spaces. Here, for $\F, \G \in \D_{{\rm adm}}^b(\k_{X \times X \times I})$, the equivalence relation $\sim$ between $\F$ and $\G$ is defined as follows: $\F \sim \G$ if and only if (i) $\F|_{t=1} \simeq \G|_{t=1}$ and (ii) there exists $\Theta \in \D_{{\rm adm}}^b(\k_{X \times X \times I \times I})$, where the $t$ denotes the coordinate of the first $I$ and $s$ denotes the coordinate of the second $I$, such that $\Theta|_{s=0} = \F$ and $\Theta|_{s=1} = \G$. We call {\it $\F$ is homotopic to $\G$} if the second condition (ii) is satisfies. What Proposition 4.3 in \cite{Zha20} essentially proves is that if $\phi$ is homotopic to $\psi$ through contact isotopies, then,  $\K_{\phi}|_{t=1} = \K_{\psi}|_{t=1}$, and there exists a homotopy of sheaves as above in $\D_{\rm adm}^b(\k_{X \times X \times I})$ from $\K_{\phi}$ to $\K_{\psi}$. In other words, $\K_{\phi} \sim K_{\psi}$. The proof of this is similar to the proof of the uniqueness of the sheaf quantization, which eventually comes from the fact that $I$ is contractible (so $I \times I$ is also contractible). The embedding in (\ref{emb-sheaf-2}) is then given by $\sigma: [\phi] \mapsto [\K_{\phi}]$.

Let $M = S_g^*X$ and $\xi_{\rm can}$ be the canonical contact structure on $M$. We call the subcategory {\it $\D_{{\rm adm}}^b(\k_{X \times X \times I})/\sim$ orderable} if the partial order $\geq_s$ descends to a well-defined partial order to $\D_{{\rm adm}}^b(\k_{X \times X \times I})/\sim$. The following glossary summarizes the correspondences between the elements from contact geometry and the elements from derived category. 

\medskip

\begin{center}
\begin{tabular}{ P{1em} | P{5cm} | P{5.5cm}  } 
\hline
& C (= contact geometry) & D (= derived category) \\ 
\hline
1 & $\phi$, \,\,$\phi^{-1}$, \,\,$\mathds{1}$ & $\K_{\phi}$, \,\,$\K_{\phi}^{-1}$, \,\,$\k_{\Delta \times I}$ \\ 
\hline
2 & $\phi \sim \psi$ & $\K_{\phi} \sim \K_{\psi}$ \\ 
\hline
3& ${\rm PCont}_0(M,\xi)$ & $\D_{{\rm adm}}^b(\k_{X \times X \times I})$ \\ 
\hline
4& ${\rm PCont_+}(M,\xi)$ & $\D_{{\rm adm} \cap {\rm dom}}^b(\k_{X \times X \times I})$ \\ 
\hline
5& $\widetilde{{\rm Cont}}_0(M, \xi)$ & $\D_{{\rm adm}}^b(\k_{X \times X \times I})/\sim$ \\ 
\hline
6& \mbox{$\geq_o$ on ${\rm PCont}_0(M,\xi)$} & \mbox{$\geq_s$ on $\D_{{\rm adm} \cap {\rm dom}}^b(\k_{X \times X \times I})$}  \\ 
\hline
7& \mbox{$(M, \xi_{\rm can})$ is orderable}  & \mbox{$\D_{{\rm adm}}^b(\k_{X \times X \times I})/\sim$ is orderable} \\ 
\hline
8& any non-negative contractible loop of contactomorphisms of $(M, \xi_{\rm can})$ is trivial & Conjecture \ref{conj-anti-sym} \\ 
\hline
\end{tabular}
\end{center}

\medskip

Some explanations are in order. The correspondence between 1C and 1D are given by the sheaf quantization from \cite{GKS12}, which yields an embedding from 3C to 3D (as well as an embedding from 4C to 4D as in Theorem \ref{prop-sheaf-order}). Meanwhile, as elaborated above, 2C implies 2D due to Proposition 4.3 in \cite{Zha20}, which yields an embedding from 5C to 5D as in (\ref{emb-sheaf-2}). The partial order $\geq_o$ in 6C is the standard partial 
order defined in \cite{EP00}, which will be elaborated in details in subsection \ref{subsec-pf-g1-prop}, while the relation $\geq_s$ in 6D is the partial order defined and confirmed in Definition \ref{sheaf-order} and Theorem \ref{prop-sheaf-order}, respectively. Due to Proposition 2.1.A in \cite{EP00}, 7C and 8C are equivalent. The 8D, formulated as Conjecture \ref{conj-anti-sym} below, can be viewed as a sheaf-version of 8C. Moreover, following the same idea as in the proof of Proposition 2.1.A in \cite{EP00}, one can easily show that 7D and 8D are equivalent. As a matter of fact, since $M = S^*_gX$, \cite{CN10} proves that 7C indeed holds (thus 8C also holds). Proving 7C is not directly via 8C; instead, the proof in \cite{CN10} requires more advanced machinery. On the other hand, it is unknown to us that the corresponding 7D (thus also 8D) holds. Since \cite{GKS12} also proves 7C from a perspective of the sheaf quantization, we expect that the proof of Theorem 4.13 in \cite{GKS12}, in particular some homotopy version of its Proposition 4.8, can be helpful.

\begin{conj} \label{conj-anti-sym} Suppose that $\F \in \D_{{\rm adm}}^b(\k_{X \times X \times I})$ satisfies the conditions (i) $\F|_{t=1} = \k_{\Delta}$, (ii) $SS(\F) \subset T^*_{\{\tau\leq0\}} (X \times X \times I)$ and (iii) $\F$ is homotopic to $\k_{\Delta \times I}$. Then $\F \simeq \k_{\Delta \times I}$. \end{conj}

Recall that $G_+$ is defined in (\ref{dfn-G+1}), the set of all dominants in $\widetilde{{\rm Cont}}_0(M, \xi)$. Assume that Conjecture \ref{conj-anti-sym} holds, then the following result is immediate. 

\begin{cor} [Assume Conjecture \ref{conj-anti-sym}] The embedding defined in (\ref{emb-sheaf-2}) restricts to a monotone decreasing embedding $[\sigma]: (G_+, d_{\geq_o}) \to (\mathcal D_{{\rm adm} \cap {\rm dom}}^b(\k_{X \times X \times I})/\sim, d_{\geq_s})$, where $d_{\geq_o}$ and $d_{\geq_s}$ are the pseudo-metrics induced from the corresponding relative growth rates of the partial orders $d_{\geq_o}$ and $d_{\geq_s}$, respectively. \end{cor}

\section{Proofs} \label{ssec-results}
In this section, we will provide the necessary background of the concepts appeared in the main results in the introduction, and their detailed proofs will also be presented.  

\subsection{Proof of Theorem \ref{G1-prop}} \label{subsec-pf-g1-prop}
First, let us recall some standard background on the (non)-orderability of contact manifolds (\cite{EP00}). Let $(M, \xi)$ be a closed contact manifold with a co-oriented contact structure. Denote by ${\rm PCont}_0(M, \xi)$ the set of paths of contactomorphisms on $(M, \xi)$ which are parametrized by $[0,1]$ and start from $\mathds{1}$. Each element in ${\rm PCont}_0(M, \xi)$ is called a contact isotopy. Fix a contact 1-form $\alpha \in \Omega^1(M)$, then each contact isotopy $\phi = \{\phi_t\}_{t \in [0,1]} \in {\rm PCont}_0(M, \xi)$ is generated by a unique  contact Hamiltonian function $h_{\alpha}(\phi): [0,1] \times M \to \R$. We call $\phi$ {\it positive (or non-negative)} if the corresponding $h_{\alpha}(\phi)$ is pointwisely positive (or non-negative). Observe that being positive (or non-negative) is independent of the choice of the contact 1-forms. Denote by $SM = \R_+ \times M$ the symplectization of $(M, \xi)$ and $s$ denotes its $\R_+$-coordinate, then a contact Hamiltonian function $h_{\alpha}(\phi)$ lifts to a 1-homogenous (symplectic) Hamiltonian function $H_{\alpha}(\phi) = s\cdot h_{\alpha}(\phi)$ with respect to the symplectic structure $\omega = d(s\pi_M^*\alpha)$ on $SM$, where $\pi_M: SM \to M$ is the projection onto $M$. 

Now, consider $\widetilde{\rm Cont}_0(M, \xi)$ defined by ${\rm PCont}_0(M, \xi)/{\sim}$, where, for two elements $\phi, \psi \in {\rm PCont}_0(M, \xi)$, $\phi \sim \psi$ if and only if $\phi_1 = \psi_1$ and they are homotopic with fixed endpoints through elements in ${\rm PCont}_0(M, \xi)$. 
We call an element $[\phi] \in \widetilde{\rm Cont}_0(M, \xi)$ {\it positive (or non-negative)} if it admits a positive (or non-negative) representative. Due to Criterion 1.2.C in \cite{EP00}, a contact manifold $(M, \xi)$ is {\it non-orderable} if $[\mathds 1] \in \widetilde{\rm Cont}_0(M, \xi)$, otherwise $(M, \xi)$ is {\it orderable}. A standard example of a non-orderable contact manifold is $(S^{2n-1}, \xi_{\rm std})$ for $n \geq 2$, and a standard example of an orderable contact manifold is $(S_g^*X, \xi_{\rm can})$, the unit co-sphere bundle of a closed manifold $X$ with respect to a fixed Riemannian metric $g$ on $X$ (see Corollary 9.1 in \cite{CN10}). 

When $(M, \xi)$ is orderable, for brevity denote $G := \widetilde{\rm Cont}_0(M, \xi)$ and
\begin{equation} \label{dfn-G+1}
G_+ := \{[\phi] \in G\,| \, \mbox{$[\phi]$ is positive}\}.
\end{equation}
It is readily to check that this $G$ can be endowed with a partial order $\geq$ defined as follows. Fix any contact 1-form $\alpha$ on $M$. For any two elements $[\phi], [\psi] \in G$, define 
\begin{equation} \label{po-ep}
[\phi] \geq_o [\psi]\,\,\,\,\mbox{if and only if} \,\,\,\,\mbox{$h_{\alpha}(\phi) \geq h_{\alpha}(\psi)$ pointwisely}
\end{equation}
for some representatives (without loss of generality, again denoted by) $\phi$ and $\psi$ of $[\phi]$ and $[\psi]$, respectively (see Proposition 1.4.B in \cite{EP00}). This condition is equivalent to the condition that their (symplectic) Hamiltonian lifts satisfy $H_{\alpha}(\phi) \geq H_{\alpha}(\psi)$ pointwisely on $[0,1] \times SM$. Note that the partial order $\geq_o$ is also bi-invariant under the group structure of $G$. Moreover, in terms of $\geq_o$, elements in $G_+$ defined in (\ref{dfn-G+1}) are exactly the dominants of $G$. The relative growth rate $\gamma_{\geq_o}([\phi], [\psi])$ is studied in (\cite{EP00}), and $(G_+, d_{\geq_o})$ is a pseudo-metric space, where $d_{\geq_o}$ is defined as in (\ref{dfn-pm}). Recall that the relation $\geq_+$ is defined in (\ref{order-2}) between elements in the equivalence classes of positive paths from $\overrightarrow{G}$. Now, we are ready to give the proof of Theorem \ref{G1-prop}. 

\begin{proof} [Proof of Theorem \ref{G1-prop}] The reflexive property of the relation $\geq_+$ is obvious. Let us show the transitive property. Take $[\phi], [\psi], [\theta] \in \overrightarrow{G}$, and assume that $[\phi] \geq_+ [\psi]$ and $[\psi] \geq_+ [\theta]$. Then $[\phi] \geq_+ [\theta]$ if $[\phi] = [\psi]$ or $[\psi] = [\theta]$. Otherwise, there exist $[\xi], [\zeta] \in \overrightarrow{G}$ such that $[\phi] = [\psi][\xi]$ and $[\psi] = [\theta][\zeta]$. Then $[\phi] = [\theta]([\zeta][\xi])$ where $[\zeta][\xi] \in \overrightarrow{G}$, which implies that $[\phi] \geq_+ [\theta]$ by definition. Now, let us show the anti-symmetric property. For $[\phi], [\psi] \in \overrightarrow{G}$, if $[\phi] \geq_+ [\psi]$ and $[\psi] \geq_+ [\phi]$, then by the same argument as above, either $[\phi] =  [\psi]$ (then we obtain the desired conclusion) or there exists some non-trivial element $[\theta] \in \overrightarrow{G}$ such that $[\phi] = [\phi] [\theta]$. Therefore, if $[\phi] \neq [\psi]$, then $\pi([\phi]) = \pi([\phi]) \pi([\theta])$, where $\pi: \overrightarrow{G} \to \widetilde{\rm Cont}_0(M, \xi)$ is the canonical projection. Since ${\widetilde{\rm Cont}}_0(M, \xi)$ is a group, then one gets $\pi([\theta]) = \mathds{1}$. This contradicts our hypothesis that $(M, \xi)$ is orderable. 

In order to prove the bi-invariant property, observe that $[\phi] \geq_+ [\psi]$ if and only if $[\phi] = [\psi]$ or $[\phi] = [\xi] [\psi]$ (left multiplication) for some $[\xi] \in \overrightarrow{G}$. In fact, if $[\phi] = [\psi]$, then nothing needs to be proved. If $[\phi] \neq [\psi]$, then there exists some $[\theta] \in \overrightarrow{G}$ such that $[\phi] = [\psi] [\theta]$. Now, consider path $\{\xi_t\}_{t \in [0,1]} : = \{\psi_1 \theta_t \psi_1^{-1}\}_{t \in [0,1]}$, and we claim that $\{\xi_t \psi_t\}_{t \in [0,1]} \sim \{\psi_t \theta_t\}_{t \in [0,1]}$. Then denote by $\xi = \{\xi_t\}_{t \in [0,1]}$, we know that $[\phi] = [\psi][\theta] = [\xi] [\psi]$. To prove our claim, similarly to the proof of the formula (84) in \cite{EKP06}, consider new paths
\begin{equation*}
\psi_t' = \left\{ \begin{array}{llr} \psi_{2t} & \mbox{for} & t \in [0, 1/2]\\ \psi_1 & \mbox{for} & t \in [1/2, 1] \end{array} \right. \,\,\,\,\,\mbox{and}\,\,\,\,\, \theta_t' = \left\{ \begin{array}{llr} \mathds{1}& \mbox{for} & t \in [0, 1/2]\\ \theta_{2t-1} & \mbox{for} & t \in [1/2, 1] \end{array} \right..
\end{equation*}
Then $\{\psi_t\}_{t\in [0,1]}$ is homotopic with fixed endpoints to $\{\psi'_t\}_{t \in [0,1]}$, and $\{\theta_t\}_{t\in [0,1]}$ is homotopic with fixed endpoints to $\{\theta'_t\}_{t \in [0,1]}$. Meanwhile, all the intermediate paths, i.e., except $\{\psi'_t\}_{t \in [0,1]}$ and $\{\theta'_t\}_{t\in [0,1]}$ themselves, can be chosen to be positive. Moreover, for any $t \in [0,1]$, $\psi_t' \theta_t' = \psi_1\theta_t'\psi_1^{-1} \psi_t'$. Therefore, by a ``smooth the corner'' argument, we obtain the desired homotopy. Next, for $[\phi], [\psi] \in \overrightarrow{G}$ such that $[\phi] \geq_+ [\psi]$ and for any $[\zeta] \in \overrightarrow{G}$, by definition, either $[\zeta] [\phi] = [\zeta] [\psi]$ or $[\zeta][\phi] = [\zeta] [\psi] [\theta]$ for some $[\theta] \in \overrightarrow{G}$, which implies that $[\zeta][\phi] \geq_+ [\zeta][\psi]$. On the other hand, using the claim above, $[\phi] \geq_+ [\psi]$ implies that either $[\phi] = [\psi]$ or $[\phi] = [\xi][\psi]$ for some $[\xi] \in \overrightarrow{G}$. Then either $[\phi][\zeta] = [\psi][\zeta]$ or $[\phi] [\zeta] = [\xi] [\psi] [\zeta]$, which also implies that $[\phi][\zeta] \geq_+ [\psi][\zeta]$. 

Finally, let us show the monotonicity of the projection $\pi: \overrightarrow{G} \to \widetilde{\rm Cont}_0(M, \xi)$. For $[\phi], [\psi] \in \overrightarrow{G}$ such that $[\phi] \geq_+[\psi]$, if $[\phi] = [\psi]$, then there is nothing to prove. If $[\phi] \neq [\psi]$, then $[\phi] = [\psi] [\theta]$ for some $[\theta] \in \overrightarrow{G}$. Fix a contact 1-form $\alpha$ on $M$. Assume $[\psi]$ is represented by a positive path $\psi = \{\psi_t\}_{t \in [0,1]}$, and $[\theta]$ is represented by a positive path $\theta = \{\theta\}_{t \in [0,1]}$. Denote by $h_{\alpha}(\psi)$ and $h_{\alpha}(\phi)$ the positive contact Hamiltonian functions. Note that the representatives of a class in $\overrightarrow{G}$ are also representatives of its image under $\pi$ in ${\widetilde{\rm Cont}}_0(M, \xi)$. Therefore, for $\pi([\phi])$, one can choose the representative $\{\psi_t \theta_t\}_{t \in [0,1]}$. One can easily check that the corresponding contact Hamiltonian functions satisfy $h_{\alpha}(\psi \theta) = h_{\alpha}(\psi) + (e^{g_{\alpha}(\psi)}\cdot h_{\alpha}(\theta)) \circ \psi_t^{-1}$, where $g_{\alpha}(\psi)$ is the conformal factor of $\psi$ with respect to $\alpha$. By the positivity of $h_{\alpha}(\theta)$, we know $h_{\alpha}(\psi \theta) \geq h_{\alpha}(\psi)$ pointwisely, and then $\pi([\phi]) \geq_+ \pi([\psi])$ by (\ref{po-ep}). 
\end{proof}

When $(M, \xi)$ is non-orderable, we claim that there does not exist such an analogue result as in Theorem \ref{G1-prop}. Recall that the condition $(M, \xi)$ is non-orderable is equivalent to $\mathds{1} \in \pi(\overrightarrow{G})$, where $\pi: \overrightarrow{G} \to {\widetilde{\rm Cont}}_0(M, \xi)$ is the natural projection. Denote by $\overrightarrow{G}_0 = \pi^{-1}(\mathds{1})$. There is an interesting observation from Corollary 6.6 in \cite{EKP06} that $\overrightarrow{G}_0$ contains a unique ``stable'' elements denoted by $\theta_{\rm st}$ in that sense that any element $\theta \in \overrightarrow{G}_0$ admits a power $N \in \N$ such that $\theta^N = \theta_{\rm st}$. This makes the structure of $\overrightarrow{G}$ delicate. Consider the following subset of $\overrightarrow{G}$, 
\begin{equation} \label{dfn-Gamma}
\Gamma = \{ x \in \overrightarrow{G} \,| \, x \theta_{\rm st} = x\}.
\end{equation}
The partial order $\geq_+$ defined in (\ref{order-2}) does not apply to $\overrightarrow{G}$, but, due to (iv) in Proposition 6.10, it does define a genuine partial order on $\overrightarrow{G} \backslash \Gamma$. Note that $\overrightarrow{G} \backslash \Gamma$ is not always empty (see Proposition 6.13 in \cite{EKP06} for the case where $(M, \xi) = (S^{2n-1}, \xi_{\rm std})$ and $n \geq 2$). However, we observe that the relative growth rate is not well-defined on $\overrightarrow{G} \backslash \Gamma$ mainly due to (iii) in Proposition 6.9 in \cite{EKP06}. Explicitly, $\overrightarrow{G}$ is an epigroup with respect to $\Gamma$ in the sense that every element $x \in \overrightarrow{G}$ admits a power $N \in \N$ such that $x^N \in \Gamma$. Therefore, $\overrightarrow{G} \backslash \Gamma$ is not even closed under the semi-group multiplication of $\overrightarrow{G}$. If one wants to relax the discussion of $\geq_+$ to the entire $\overrightarrow{G}$, then we claim that the relative growth rate $\gamma_{\geq _+} \equiv 0$, so $d_{\geq _+}$ is still not well-defined. In fact, for any sufficiently large $k \in \N$ such that $x^k \in \Gamma$, we can take an arbitrarily large $l \in \N$ such that $y^{l} \in \Gamma$. Then, by (ii) in Proposition 6.10 in \cite{EKP06}, $x^k \geq y^{l}$. Therefore, $\rho^+_{\geq _+}(x,y) = 0$. Similarly, $\rho^-_{\geq_+}(y,x) =0$, which implies that the relative growth rate $\gamma_{\geq_+}(x,y)= 0$. 

\medskip

In this subsection, we also want to provide a useful observation of the relation $\geq_o$. Explicitly, on an orderable contact manifold, the relative growth rate derived from the partial order $\geq_o$ is closely related to certain norms on $G$. Recall that (see Remark 7 in \cite{She17}) a pseudo-norm $\nu$ defined on a group $G$ is called {\it quasi-conjugate invariant} if for each $b \in G$, there exists a constant $C(b)$ such that $\nu(bab^{-1}) \leq C(b) \cdot \nu(a)$ for any $a \in G$. Let $G = \widetilde{\rm Cont}_0(M, \xi)$ and $G_+$ be the subset defined in (\ref{dfn-G+1}) above. For any $[\phi] \in G_+$ and any $[\psi] \in G$, define 
\begin{equation} \label{nu+}
\nu_{[\phi]}^+([\psi]) := \min\{k \in \Z\,| \, [\phi]^k \geq_o [\psi]\},
\end{equation}
and
\begin{equation} \label{nu-}
\nu_{[\phi]}^-([\psi]) := \max\{l \in \Z\,| \, [\psi] \geq_o [\phi]^l\}.
\end{equation}
Then define
\begin{equation} \label{nu}
\nu_{[\phi]}([\psi]) := \max\left\{ |\nu_{[\phi]}^+([\psi])|, |\nu_{[\phi]}^-([\psi])| \right\}.
\end{equation}
We have the following proposition. 
\begin{prop} \label{gamma-nu}
Let $(M, \xi)$ be a closed orderable contact manifold, $G = \widetilde{\rm Cont}_0(M, \xi)$ and $G_+$ be defined in (\ref{dfn-G+1}). Then we have the following two properties.
\begin{itemize}
\item[(a)] For any element $[\phi] \in G_+$, the function $\nu_{[\phi]}: G \to \R$ defines a quasi-conjugate invariant norm on $G$.
\item[(b)] For any $[\psi] \in G$, the stabilization of $\nu_{[\phi]}([\psi])$ satisfies the following relation,
\begin{equation} \label{stab-gamma-nu}
\lim_{\ell \to \infty} \frac{\nu_{[\phi]}([\psi]^\ell)}{\ell} = \max\left\{ |\rho^+_{\geq_o}([\phi], [\psi])|, |\rho^+_{\geq_o}([\phi], [\psi]^{-1})|\right\},
\end{equation}
\end{itemize}
where $\rho^+_{\geq_o}$ is defined as in (\ref{dfn-rgr}) with respect to the partial order $\geq_o$ on $G$.
\end{prop}

\begin{remark} When $[\phi] \in G_+$ is a homotopy class represented by a loop, $\nu_{[\phi]}$ is in fact a conjugate-invariant norm. This easily comes from the fact that any $[\phi] \in \pi_1({\rm Cont}(M, \xi), \mathds{1})$ lies in the center of $G$. Then, by the bi-invariant property of the partial order $\geq_o$, both $\nu_{[\phi]}^+$ and $\nu_{[\phi]}^-$ are conjugate-invariant. If, furthermore, when $[\phi]$ is represented by a 1-periodic Reeb flow (if it exists), then the definition (\ref{nu+}) and (\ref{nu-}) are the functions $\nu_+$ and $\nu_-$ defined in subsection 2.2 in \cite{FPR18}. Therefore, the item (a) in Proposition \ref{gamma-nu} can be regarded as a generalization of Theorem 2.15 in \cite{FPR18}.  \end{remark}

\begin{remark} Let $G$ be a group endowed with a norm $\nu$. Recall that (see subsection 1.3 in \cite{Pol06}) an element $g \in G$ is called {\it undistorted with respect to $\nu$} if the stabilization $\lim_{\ell \to \infty} \frac{\nu(g^\ell)}{\ell} >0$, otherwise it is called {\it distorted with respect to $\nu$}. In general, it is difficult to determine whether an element is (un)distorted or not. The relation (\ref{stab-gamma-nu}) in (b) in Proposition \ref{gamma-nu} implies that an element $[\psi] \in G$ is undistorted with respect to $\nu_{[\phi]}$ if and only if at least one of $\rho^+_{\geq_o}([\phi], [\psi])$ and $\rho^+_{\geq_o}([\phi], [\psi]^{-1})$ is non-zero. For instance, if $[\psi]$ or $[\psi]^{-1}$ is in $G_+$, then due to the inequality (1.1.A) in \cite{EP00}, $\rho^+_{\geq_o}([\phi], [\psi]) \neq 0$ or $\rho^+_{\geq_o}([\phi], [\psi]^{-1}) \neq 0$. This implies the undistortedness of $[\psi]$ or $[\psi]^{-1}$ with respect to $\nu_{[\phi]}$. 
\end{remark}

\begin{proof} [Proof of Proposition \ref{gamma-nu}] Let us prove (a) first. The value of $\nu_{[\phi]}$ is obviously non-negative. Moreover, if $\nu_{[\phi]}([\psi]) = 0$, then $\nu_{[\phi]}^+([\psi])=0$ and $\nu_{[\phi]}^-([\psi])=0$, which implies that ${\mathds 1} \geq_o [\psi] \geq_o {\mathds 1}$. Then by the reflexivity of the partial order $\geq_o$, we know that $[\psi] = \mathds 1$. Therefore, $\nu_{[\phi]}$ is non-degenerate. Now, for $[\psi]^{-1}$, 
\begin{align*}
\nu_{[\phi]}^+([\psi]^{-1})& = \min\{k \in \Z\,| \, [\phi]^k \geq_o [\psi]^{-1} \}\\
& = \min\{k \in \Z\,| \, [\psi] \geq_o [\phi]^{-k} \}\\
& = \min\{-l \in \Z\,|\, [\psi] \geq_o [\phi]^{l}\} \,\,\,\,\,\, (\mbox{set $l = -k$})\\
& = - \max\{ l \in \Z \,|\, [\psi] \geq_o [\phi]^l\} = - \nu_{[\phi]}^-([\psi]),
\end{align*}
where the second equality comes from the fact that the partial order $\geq_o$ is bi-invariant. Similarly, one can show that $\nu^{-}_{[\phi]}([\psi]^{-1}) = - \nu_{[\phi]}^{+}([\psi])$. Therefore, 
\begin{align*}
\nu_{[\phi]}([\psi]^{-1}) & = \max\left\{|\nu_{[\phi]}^+([\psi]^{-1})|, |\nu^{-}_{[\phi]}([\psi]^{-1})|\right\} \\
& = \max\left\{|- \nu_{[\phi]}^-([\psi])|, |- \nu_{[\phi]}^{+}([\psi])|\right\} = \nu_{[\phi]}([\psi]).
\end{align*}
Now, observe that, for any $\sigma \geq 0$, 
\begin{equation} \label{tri}
\nu_{[\phi]}([\psi]) \leq \sigma \,\,\,\,\mbox{if and only if} \,\,\,\, [\phi]^{\sigma} \geq_o [\psi] \geq_o [\phi]^{-\sigma}.
\end{equation}
This comes from definitions (\ref{nu+}) and (\ref{nu-}), as well as the fact that $\nu^{-}_{[\phi]} \leq \nu^{+}_{[\phi]}$ (by the transitivity of the partial order $\geq_o$). Next, for $[\psi_1], [\psi_2] \in G$, denote by $\sigma_1 = \nu_{[\phi]}([\psi_1])$ and $\sigma_2 = \nu_{[\phi]}([\psi_2])$. By (\ref{tri}), $[\phi]^{\sigma_1} \geq_o [\psi_1] \geq_o [\phi]^{-\sigma_1}$ and $[\phi]^{\sigma_2} \geq_o [\psi_2] \geq_o [\phi]^{-\sigma_2}$. Then 
\[ [\phi]^{\sigma_1 + \sigma_2} \geq_o [\psi_1][\psi_2] \geq_o [\phi]^{-(\sigma_1+\sigma_2)}.\]
Therefore, $\nu_{[\phi]}([\psi_1][\psi_2])\leq \sigma_1 + \sigma_2 = \nu_{[\phi]}([\psi_1]) + \nu_{[\phi]}([\psi_2])$. 
Thus, we conclude that $\nu_{[\phi]}$ is a norm. 

Now, fix an element $[\mu] \in G$, and consider $\nu_{[\phi]}([\mu] [\psi] [\mu]^{-1})$. By definition, 
\begin{align*} 
\nu^+_{[\phi]}([\mu] [\psi] [\mu]^{-1}) & = \min\{k \in \Z \,| \, [\phi]^k \geq_o [\mu] [\psi] [\mu]^{-1}\}\\
& = \min\{k \in \Z \,| \, [\mu]^{-1} [\phi]^k [\mu] \geq_o [\psi]\}\\
& = \min\{k \in \Z \,| \, ([\mu]^{-1} [\phi] [\mu])^k \geq_o [\psi]\} = \nu^+_{[\mu]^{-1}[\phi][\mu]}([\psi]). 
\end{align*} 
A similar conclusion holds for $\nu^-_{[\phi]}([\mu] [\psi] [\mu]^{-1})$. Therefore, $\nu_{[\phi]}([\mu] [\psi] [\mu]^{-1}) =   \nu_{[\mu]^{-1}[\phi][\mu]}([\psi])$.

Next, fix a contact 1-form $\alpha$ on $M$. Assume that $\nu_{[\phi]}([\psi]) = \sigma$, then the relation (\ref{tri}) follows, which implies that there exist a representative $\phi^{(\sigma)}$ of $[\phi]^{\sigma}$ (so $(\phi^{(\sigma)})^{-1}$ is a representative of $[\phi]^{-\sigma}$) and representatives $\psi', \psi''$ of $[\psi]$ such that their corresponding contact Hamiltonian functions satisfy the relations
\[ h_{\alpha}(\phi^{(\sigma)}) \geq h_{\alpha}(\psi') \,\,\,\,\mbox{and} \,\,\,\, h_{\alpha}(\psi'') \geq h_{\alpha}((\phi^{(\sigma)})^{-1}).\]
We can assume that $h_{\alpha}(\phi^{(\sigma)})>0$ since $[\phi] \in G_+$, and then $h_{\alpha}((\phi^{(\sigma)})^{-1})<0$. Choose any representative $\mu$ of $[\mu]$, then $\mu^{-1} \phi^{(\sigma)} \mu$ represents the class $[\mu]^{-1}[\phi]^{\sigma} [\mu]$. Denote by $g_{\alpha}(\mu)$ the conformal factor of the path $\mu$ under the contact 1-form $\alpha$, one can easily check that the contact Hamiltonian function of the path $\mu^{-1} \phi^{(\sigma)} \mu$ is $e^{-g_{\alpha}(\mu)} (h_{\alpha}(\phi^{(\sigma)}) \circ \mu)$, which is, in particular, positive. Due to Lemma 3.1.A in \cite{EP00}, there exists a representative of $[\mu]^{-1}[\phi]^{\sigma} [\mu]$ such that the corresponding contact Hamiltonian function, denoted by $f_{\mu}(t,x)$, is 1-periodic and positive. Meanwhile, denote the contact Hamiltonian function of $\mu^{-1} (\phi^{(\sigma)})^{-1} \mu$ by $\overline{f}_{\mu}(t,x)$ which is 1-periodic and negative. Meanwhile, we can modify $h_{\alpha}(\psi')$ near $t=0, 1$ to be $\tilde h_{\alpha}(\psi')$ such that $\tilde h_{\alpha}(\psi')=0$ near $t=0,1$ and $\tilde h_{\alpha}(\psi')$ still generates a representative of the class $[\psi']$. Therefore, we can regard $\tilde h_{\alpha}(\psi')$ as a function defined on $\R/\Z \times M$ (by identifying the endpoints $t=0,1$). The same conclusion holds for $h_{\alpha}(\psi'')$, and denote by $\tilde{h}_{\alpha}(\psi'')$ the modified function defined on $\R/\Z \times M$.

By the positivity of $f_{\mu}(t,x)$ and the negativity of $\overline{f}_{\mu}(t,x)$, there exists a sufficiently large constant $C>1$, which only depends on $[\phi]$ and $[\mu]$, such that for any $(t,x) \in \R/\Z \times M$, 
\[ (Ce^{g_{\alpha}(\mu)}) \cdot f_{\mu}(t,x) \geq \max_{(t,x) \in \R/\Z \times M} \tilde h_{\alpha}(\psi'),\]
and 
\[ \min_{(t,x) \in \R/\Z \times M} \tilde h_{\alpha}(\psi'') \geq (Ce^{g_{\alpha}(\mu)}) \cdot \overline{f}_{\mu}(t,x).\]
Denote by $C(\mu) := \ceil*{C e^{g_{\alpha}(\mu)}}$. We know that $C(\mu) \cdot f_{\mu}(C(\mu) t,x)$ generates a representative of $([\mu]^{-1}[\phi]^{\sigma} [\mu])^{C(\mu)} = ([\mu]^{-1} [\phi][\mu])^{C(\mu)\sigma}$, and $C(\mu) \cdot f_{\mu}(C(\mu) t,x) \geq \tilde h_{\alpha}(\psi')$ pointwisely on $\R/\Z \times M$. Therefore, $([\mu]^{-1} [\phi][\mu])^{C(\mu)\sigma} \geq_o [\psi]$. Similarly, $C(\mu) \cdot \overline{f}_{\mu}(C(\mu) t,x)$ generates a representative of $([\mu]^{-1}[\phi]^{-\sigma} [\mu])^{C(\mu)} = ([\mu]^{-1} [\phi][\mu])^{-C(\mu)\sigma}$, and $\tilde h_{\alpha}(\psi'') \geq C(\mu) \cdot \overline{f}_{\mu}(\sigma(\mu) t,x)$ pointwisely on $\R/\Z \times M$. Therefore, $[\psi] \geq_o ([\mu]^{-1} [\phi][\mu])^{-C(\mu)\sigma}$. Together, we get 
\[ ([\mu]^{-1} [\phi][\mu])^{C(\mu)\sigma} \geq_o [\psi] \geq_o ([\mu]^{-1} [\phi][\mu])^{-C(\mu)\sigma}. \]
Therefore,
\begin{align*}
\nu_{[\phi]}([\mu] [\psi] [\mu]^{-1}) & = \nu_{[\mu]^{-1}[\phi][\mu]}([\psi]) \leq C(\mu)\sigma = C(\mu)\cdot \nu_{[\phi]}([\psi]). 
\end{align*}
This holds for any $[\psi] \in G$, then we conclude that the norm $\nu_{[\phi]}$ is quasi-conjugate invariant.

Finally, let us proof (b). By definition, 
\begin{equation} \label{stab-1}
\lim_{\ell \to \infty} \frac{\nu_{[\phi]}^+([\psi]^\ell)}{\ell} = \lim_{\ell \to \infty} \frac{\min\{k \in \Z \,| \, [\phi]^k \geq_o [\psi]^\ell\}}{\ell} = \rho_{\geq_o}^+([\phi], [\psi]). 
\end{equation}
Due to the discussion above, we know that $\nu_{[\phi]}^-([\psi]) = -\nu_{[\phi]}^+([\psi]^{-1})$. Then 
\begin{equation} \label{stab-2}
\lim_{\ell \to \infty} \frac{\nu_{[\phi]}^-([\psi]^\ell)}{\ell} = \lim_{\ell \to \infty}  \frac{-\nu_{[\phi]}^+(([\psi]^{-1})^\ell)}{\ell}  = -\rho_{\geq_o}^+([\phi], [\psi]^{-1}). 
\end{equation}
Therefore, 
\begin{align*}
\lim_{\ell \to \infty} \frac{\nu_{[\phi]}([\psi]^\ell)}{\ell} & = \max\left\{ \left|\lim_{\ell \to \infty} \frac{\nu_{[\phi]}^+([\psi]^\ell)}{\ell}\right|, \left|\lim_{\ell \to \infty} \frac{\nu_{[\phi]}^-([\psi]^\ell)}{\ell}\right|\right\}\\
& =\max\left\{ |\rho_+([\phi], [\psi])|, |\rho_+([\phi], [\psi]^{-1})|\right\}.
\end{align*}
Therefore, we get the desired conclusion.
\end{proof}

\subsection{Proof of Theorem \ref{lemma-ccbm-property}} \label{subsec-pf-ccbm}

Let us first briefly recall the definitions of various symplectic-style Banach-Mazur distances. For two open star-shaped domains $U, V$ of a Liouville manifold $(W, \omega, L)$, where $L$ is the Liouville vector field, a quantitative and symplectic way to compare $U$ and $V$ is via the coarse symplectic Banach-Mazur distance (\cite{GU19}, \cite{Ush18}). Explicitly, it is defined by
\begin{equation} \label{dfn-csbm}
d_c(U, V):= \inf\left\{\ln C >1\,\bigg| \,(\phi, \psi) \,\,(U \xhookrightarrow{\phi} CV\,\,\mbox{and}\,\, V \xhookrightarrow{\psi} CU)\right\},
\end{equation}
where (i) the notation ``$\xhookrightarrow{\phi}$'' between two star-shaped domains $U$ and $V$ means that there exists a Hamiltonian isotopy $\phi = \{\phi_t\}_{t \in [0,1]}$ defined on $W$ such that $\phi_0 = \mathds{1}$ and $\phi_1(U) \subset V$; (ii) the rescaling is defined by $CV := \phi_{L}^{\ln C}(V)$, where $\phi_L^t$ is the flow along the Liouville vector field $L$ (similarly to the definition of $CU$). Meanwhile, there exists a ``stronger'' version of $d_c$ called the symplectic Banach-Mazur distance $d_{\rm SBM}$ (\cite{SZ18}, \cite{Ush18}), which requires the composition $\psi \circ \phi$ in (\ref{dfn-csbm}) to be isotopic {\it inside $CU$} to the identity map of $W$ through Hamiltonian isotopies. This is one way to put the ``unknottedness'' condition of the morphisms in symplectic geometry. Obviously $d_c \leq d_{\rm SBM}$. In fact, due to the work in \cite{GU19}, for many cases the inequality ``$\leq$'' of $d_c \leq d_{\rm SBM}$ can be strict ``$<$''. Then main reason to put this unknottedness condition is to guarantee that filtered symplectic homology can be applied.

As explained in subsection \ref{section-cbm-intro} in the introduction, in the contact geometry set-up we will consider contact manifolds in the form of $W \times S^1$ where $W = (W, \omega, L)$ is a Liouville manifold. The contact structure of $W \times S^1$ is given by the (standard) contact 1-form $\alpha = \iota_L \omega + dt$, where $t$ is the coordinate of $S^1$-component. In this case, a certain rigidity of contact embeddings does appear (see Theorem 1.2 in \cite{EKP06} or Theorem 1.2' in \cite{Fra16}). We call an open domain $U \subset W \times S^1$ a {\it fiberwise star-shaped domain} if (i) $\partial \overline{U}$ is transversal to the fibers and (ii) $U \cap (W \times \{t\})$ is a star-shaped domain of $W$. For instance, let $B^{2n}(R) = \{x \in \R^{2n} \,|\, \pi |x|^2 < R\}$, and then $B^{2n}(R) \times S^1$ is a fiberwise star-shaped domain of $\R^{2n} \times S^1$. The main result in \cite{Fra16} (which generalizes the main result in \cite{EKP06}) says that for any $1< R_1 < R_2$, there does not exist any compactly supported contactomorphism $\phi$ on $\R^{2n} \times S^1$ such that $\phi(B^{2n}(R_2) \times S^1) \subset B^{2n}(R_1) \times S^1$. 

Recall that in this set-up the rescaling of a fiberwise star-shaped domain $U \subset W \times S^1$ is defined via a covering map defined in Definition \ref{dfn-cont-rescale}. Inspired by the relative growth rate, this results in the definition of the contact Banach-Mazur distance $d_{\rm CBM}$ in Definition \ref{dfn-ccbm}. Before we present the proof of Theorem \ref{lemma-ccbm-property}, let use point out an essential difference between $d_{\rm CBM}$ and $d_c$, that is, $d_{\rm CBM}$ is not always well-defined. This is deeply due to the fact that contactomorphisms do not necessarily preserve volumes, and it is in a sharp contrast with the fact that any Hamiltonian diffeomorphism preserve volumes, which implies that, for any star-shaped domain $U \subset W$, $d_c(U,U) = \ln 1 = 0$. Let us elaborate on this particular phenomenon of $d_{\rm CBM}$ from the following definition. 

\begin{dfn} \label{dfn-sq} We call a fiberwise star-shaped domain $U \subset W \times S^1$ {\it squeezable} if there exists a pair $(k,l) \in \N^2$ with $k<l$ and a compactly supported contact isotopy $\{\phi_t\}_{t \in [0,1]}$ of $W \times S^1$ such that $\phi_1(U/k) \subset U/l$. A fiberwise star-shaped domain $U \subset W \times S^1$ is called {\it non-squeezable} if it is not squeezable. \end{dfn}

From the definition, it is obvious that $d_{\rm CBM}(U,U) = 0$ or equivalently $\gamma_{\rm CBM}(U,U) = 1$ for any non-squeezable fiberwise star-shaped domain $U \subset W \times S^1$. When $W = \R^{2n}$ with $n \geq 2$, due to Theorem 1.3 in \cite{EKP06}, the contact ball $B^{2n}(R) \times S^1$ is squeezable for any $R>0$ since we can take $k, l$ sufficiently large such that both $R/k$ and $R/l$ are smaller than $1$. The concept of squeezable domains is related to the one of negligible domains (see subsection 6.3.3 in \cite{EKP06}). Recall that a fiberwise star-shaped domain $U \subset W \times S^1$ is {\it negligible} if there exists a compactly supported contact isotopy $\phi = \{\phi_t\}_{t \in [0,1]}$ such that for any open neighborhood $V$ of ${\rm Core}(W) \times S^1$, we have $\phi_1(U) \subset V$. Obviously a negligible domain $U$ is squeezable since we can take the desired pair $(k,l) \in \N^2$ to be $(1,l)$ for any large $l$, where $U/l$ is regarded as an open neighborhood of ${\rm Core}(W) \times S^1$. Therefore, due to Proposition 6.12 in \cite{EKP06}, there exist squeezable fiberwise star-shaped domains of $W \times S^1$ whenever $W$ admits a non-orderable contact hypersurface. On the other hand, the celebrated contact non-squeezing theorem in \cite{EKP06}, as well as its generalizations in \cite{Fra16} and \cite{Chiu17}, says that the contact ball $B^{2n}(R) \times S^1$ is not negligible for any $R\geq1$. Therefore, the set of all the non-squeezable fiberwise star-shaped domains of $W\times S^1$ in general strictly contains the set of all the negligible fiberwise star-shaped domains of $W\times S^1$. The following property shows that the existence of certain squeezable fiberwise star-shaped domains makes the definition of $d_{\rm CBM}$ invalid. 

\begin{prop} \label{prop-gamma-0} If there exists a squeezable {\it split} fiberwise star-shaped domain of $W \times S^1$, then $\gamma_{\rm CBM}(U,V) = 0$ for any fiberwise star-shaped domains $U, V \subset W \times S^1$. In particular, $d_{\rm CBM}$ is not well-defined. \end{prop}

\begin{proof} First, let us claim that the existence of a negligible domain implies that $\gamma(U,V) =0$ for any fiberwise star-shaped domains $U,V \subset W \times S^1$. In fact, denote by $A \subset W \times S^1$ the negligible fiberwise star-shaped domain. For the domain $U$, there exists a sufficiently large $k \in \N$ such that $U/k \subset A$. Then $U/k$ is negligible. Hence, by definition, for any arbitrarily large $l \in \N$, there exists a compactly supported contact isotopy denoted by $\phi_{k,l} = \{\phi_t\}_{t \in [0,1]}$ such that $\phi_1(U/k) \subset V/l$, which implies that $\rho_c(U,V) = 0$. The same argument yields that $\rho_c(V, U) = 0$. Then $\gamma_{\rm CBM}(U,V) =0$. Second, suppose that $\hat{U} \subset W \times S^1$ is a squeezable split domain from our hypothesis. By definition, there exists a pair $(k,l) \in \N^2$ with $k <l$, and a compactly supported contact isotopy $\phi = \{\phi_t\}_{t \in [0,1]}$ such that $\phi_1(\hat{U}/k) \subset \hat{U}/l$. Observe that we have the following relation (where the second strict-containing relation is exactly due to the hypothesis that the domain $\hat{U}$ is split),
\begin{equation} \label{inclusion}
\phi_1(\hat{U}/k) \subset \hat{U}/l \subsetneqq \hat{U}/k.
\end{equation}
Repeatedly using relation (\ref{inclusion}), we obtain a sequence of strictly decreasing subsets $\phi_1(\hat{U}/k)$, $\phi^{(2)}_1(\hat{U}/k)$, $\phi^{(3)}_1(\hat{U}/k)$, etc., where $\phi^{(p)}_1$ is the time-1 map of the composition of the contact isotopy $\phi$ for $p$ times. For any given neighborhood $O$ of ${\rm Core}(W) \times S^1$, there exists a sufficiently large $p$ such that $\phi^{(p)}_1(\hat{U}/k) \subset O$. In other words, ${\hat U}/k$ is  a negligible domain. Then our desired conclusion follows from the claim at the beginning of this proof. \end{proof}

On the other hand, the following example shows that there indeed exist non-squeezable domains for certain $W \times S^1$ (therefore Definition \ref{dfn-sq} is not empty). 

\begin{ex} Let $W  = T^* \T^n$ for $n \geq 1$. Let $U = B^n(R) \times \T^n \times S^1$ for any $R>0$. Then $U$ is a fiberwise star-shaped domain of $W \times S^1$. Then we claim that $U$ is non-squeezable. It is non-trivial to check this claim. In fact, this comes from the obstruction obtained from the contact shape invariant; see Theorem \ref{thm-stability}, the functorial property (2) in Lemma \ref{csh-functor} and Example \ref{ex-torus}. \end{ex} 

\begin{remark} We conjecture that on $W \times S^1$ where $W  = T^*N$ for any closed manifold $N$, there does {\it not} exist any squeezable fiberwise star-shaped domain. We are not able to prove this at present. In fact, we do not even know how to prove a weaker conjecture that if $W = T^*N$ as above, then $W \times S^1$ does not admit any negligible domains. In general, we know that if $W$ admits a non-orderable star-shaped hypersurface, then $W \times S^1$ admits negligible domains (see Proposition 6.12 in \cite{EKP06}), but we don't know whether the converse is true or not. \end{remark}

Recall that $\mathcal N_{W \times S^1}$ denotes the set of all the non-squeezable fiberwise star-shaped domains of $W\times S^1$. Now, let us give the proof of Theorem \ref{lemma-ccbm-property}. 

\begin{proof} [Proof of Theorem \ref{lemma-ccbm-property}] First, we claims that for any $U, V \in \mathcal N_{W \times S^1}$, $\rho_c(U, V) \rho_c(V,U) \geq 1$. In fact, for any $\ep>0$, there exist pairs $(k,l)$ and $(m, n)$ in $\N^2$, and compactly supported contact isotopies of $W \times S^1$, $\phi = \{\phi_t\}_{t \in [0,1]}$ and $\psi = \{\psi_t\}_{t \in [0,1]}$, such that 
\[ \phi_1(U/k) \subset V/l \,\,\,\,\mbox{and}\,\,\,\, \psi_1(V/m) \subset U/n,  \]
and $\frac{k}{l} \leq \rho_c(U, V) +\ep$ and $\frac{m}{n}\leq \rho_c(V, U) + \ep$. Consider the covering map $\tau_m: W \times S^1 \to W \times S^1$ defined in (\ref{cover}). Then we have the following commutative diagram\footnote{This argument is borrowed from a part of the proof of Lemma 2.3 in \cite{Fra16}.},
\begin{equation} \label{lift}
\xymatrix{
W \times S^1 \ar[r]^-{\tilde{\phi}_t} \ar[d]^-{\tau_m} \ar@{-->}[rd]^-{\hat{\phi}_t} & W \times S^1 \ar[d]^-{\tau_m}\\
W \times S^1 \ar[r]^-{\phi_t} & W \times S^1}
\end{equation}
where $\hat{\phi}_t = \phi_t \circ \tau_m$, and $\tilde{\phi}_t$ is the lift of $\phi_t$. In fact, $(\hat{\phi}_t)_*(\pi_1(W \times S^1)) = (\tau_m)_*(\pi_1(W \times S^1))$ since the restriction $\phi_t|_W$ is compactly supported in $W$ and $W$ is always non-compact. By the lifting criterion (see Proposition 1.33 in \cite{Hat02}), there exists a unique lift of $\hat{\phi}_t$, denoted by $\tilde{\phi}_t$. This holds for any $t \in [0,1]$. In particular, $\phi_1$ lifts to a morphism $\tilde{\phi}_1$ and it is a contactomorphism (see Remark \ref{rmk-cont}). This implies that $\tilde{\phi}_1(\tau_m^{-1}(U/k)) \subset \tau_m^{-1}(V/l)$. Similarly, we obtain a lift of $\psi_1$ and the relation $\tilde{\psi}_1(\tau_l^{-1}(V/m)) \subset \tau_l^{-1}(U/n)$. By definition, we know that
\[ \tilde{\phi}_1(U/mk) \subset V/ml \,\,\,\,\mbox{and}\,\,\,\, \tilde{\psi}_1(V/ml) \subset U/nl. \]
Therefore, $(\tilde{\psi}_1 \circ \tilde{\phi}_1)(U/mk) \subset U/nl$. Since $U \in \mathcal N_{W \times S^1}$ is non-squeezable, by the defining property, $mk \geq nl$, which is equivalent to the inequality $\frac{k}{l} \frac{m}{n} \geq 1$, and then $\rho_c(U,V) \rho_c(V,U) + O(\ep) \geq 1$. Let $\ep \to 0$, and we get our claim. This implies that, for any $U, V \in \mathcal N_{W \times S^1}$, 
\[ d_{\rm CBM}(U,V) = \max\{ \ln \rho_c(U,V), \ln \rho_c(V,U)\} \geq 0. \]
For any $U$, $\rho_c(U,U) = 1$ since the optimal choice of the pair $(k,l) \in \N^2$ is just $(1,1)$. This implies that $d_{\rm CBM}(U,U) = 0$. Second, the symmetry of $d_{\rm CBM}$ is obvious. Third, for $U_1, U_2, U_3 \in \mathcal N_{W \times S^1}$, if $\phi_1(U_1/k) \subset U_2/l$ and $\psi_1(U_2/m) \subset U_3/n$ for some compactly supported contact isotopies of $W \times S^1$, $\{\phi_t\}_{t \in [0,1]}$ and $\{\psi_t\}_{t \in [0,1]}$, then by the argument above, we obtain the following relations with respect to the lifting contact isotopies,  
\[ \tilde{\phi}_1(U_1/mk) \subset U_2/ml \,\,\,\,\mbox{and}\,\,\,\, \tilde{\psi}_1(U_2/ml) \subset U_3/nl. \]
Therefore, $(\tilde{\psi}_1 \circ \tilde{\phi}_1)(U_1/mk) \subset U_3/nl$. For any $\ep>0$, suppose $\frac{k}{l} \leq \rho_c(U_1, U_2) +\ep$ and $\frac{m}{n} \leq \rho_c(U_2, U_3) + \ep$, then $\rho_c(U_1, U_3) \leq \frac{k}{l} \frac{m}{n} \leq \rho_c(U_1, U_2) \rho_c(U_2, U_3) + O(\ep)$. Let $\ep \to 0$, and we get $\rho_c(U_1, U_3) \leq \rho_c(U_1, U_2) \rho_c(U_2, U_3)$. A similar conclusion holds for $\rho_c(U_3, U_1)$, then the triangle inequality of $d_{\rm CBM}$ is proved. 

For the last conclusion, it suffices to prove that $\rho_c(U, \psi(V)) = \rho_c(U,V)$, and the rest part can be confirmed in a similar way. Suppose that $\psi = \psi_1$ is the time-1 map of a compactly supported contact isotopy $\psi' = \{\psi_t\}_{t \in [0,1]}$ on $W \times S^1$. By the lifting diagram (\ref{lift}), there exists a lift contact isotopy $\tilde{\psi'} = \{\tilde{\psi}_t\}_{t \in [0,1]}$ such that for any $l \in \N$, $\tilde{\psi}_t(U/l) = \psi_t(U)/l$ for any $t \in [0,1]$. For any $\ep>0$, there exists a compactly supported contact isotopy $\phi = \{\phi_t\}_{t \in [0,1]}$ and $(k,l) \in \N^2$ such that $\phi_1(U/k) \subset V/l$ and $\frac{k}{l} \leq \rho_c(U,V) + \ep$. Then the contact isotopy $\tilde{\psi'} \circ \phi : = \{\tilde{\psi}_t \circ \phi_t\}_{t \in [0,1]}$ satisfies 
\[ (\tilde{\psi'} \circ \phi)_1 (U/k) = \tilde{\psi}_1(\phi_1(U/k)) \subset \tilde{\psi}_1(V/l) =\psi(V)/l. \]
Therefore, $\rho_c(U, \psi(V)) \leq \frac{k}{l} \leq \rho_c(U, V) + \ep$, which implies that $\rho(U, \psi(V)) \leq \rho(U, V)$ if we let $\ep \to 0$. A symmetric argument then yields the desired conclusion.\end{proof}

Moreover, recall that Proposition \ref{prop-cbm-c} provides an explicit relation between $d_c$, $d_{\rm SBM}$ on the star-shaped domains of $W$ and $d_{\rm CBM}$ on the split fiberwise star-shaped domains of $W \times S^1$, via a canonical lift from $W$ to $W \times S^1$ as shown in (\ref{lift-cont}). Here, let us provide the proof of Proposition \ref{prop-cbm-c}.

\begin{proof} [Proof of Proposition \ref{prop-cbm-c}] Denote by $C := d_c(U, V)$. By definition, up to an $\ep>0$, there exist compactly supported Hamiltonian isotopies $\phi = \{\phi_t\}_{t \in [0,1]}$ and $\psi = \{\psi_t\}_{t \in [0,1]}$ such that $\phi_1(U) \subset CV$ and $\psi_1(V) \subset CU$. Then $\{\phi_t' : = \phi_L^{-\ln C} \circ \phi_t \circ \phi_L^{\ln C}\}_{t \in [0,1]}$ is a compactly supported Hamiltonian isotopy such that $\phi_1'(\frac{1}{C} U) \subset V$. For any $\delta>0$, by the density property of rational numbers, there exists a pair $(k, l) \in \N^2$ such that $\frac{l}{k} \in (\frac{1}{C}-\delta, \frac{1}{C})$. Then $\phi_1'(\frac{l}{k} U) \subset V$. By another conjugation as above, we know that there exists a compactly supported Hamiltonian isotopy $\phi'' = \{\phi''_t\}_{t \in [0,1]}$ such that $\phi_1''(\frac{1}{k} U) \subset \frac{1}{l} V$. 

Now, lift $\phi''$ to a compactly supported contact isotopy $\tilde{\phi}''$ as in (\ref{lift-cont}), and note that for the split domain $U \times S^1$, $(U \times S^1)/k \simeq \frac{1}{k} U \times S^1$. The relation $\phi_1''(\frac{1}{k} U) \subset \frac{1}{l} V$ implies that $\rho_c(U \times S^1, V \times S^1) \leq \frac{k}{l}$. Let $\delta$ go to zero, and then $\frac{k}{l}$ approaches to $C$ (as $\frac{l}{k}$ approaches to $\frac{1}{C}$). Therefore, $\rho_c(U \times S^1, V \times S^1) \leq C$. A symmetric argument implies that $\rho_c(V \times S^1, U \times S^1) \leq C$, which leads to the desired conclusion. Finally, the second conclusion is trivial since $d_c(U,V) \leq d_{\rm SBM}(U,V)$ holds by definitions. \end{proof}

\subsection{Proof of Theorem \ref{thm-dc-lb}} \label{subsec-pf-dc-lb}
Recall that subsection \ref{subsec-cbm-forms} in the introduction defines a different version of contact Banach-Mazur distance, still denoted by $d_{\rm CBM}$, between contact 1-forms of a contact manifold $(M, \xi)$ (see Definition \ref{dfn-cbm-forms}). Moreover, when $(M, \xi)$ is Liouville fillable, each contact 1-form $\alpha$ can be used to produce a star-shaped domain of $\hat{W}$ denoted by $W^{\alpha}$, where $\hat{W}$ is the completion of a filling $W$ of $(M, \xi)$. This construction is carried out in (\ref{geo-alpha}). Theorem \ref{thm-dc-lb} establishes a relation between $d_c$ and $d_{\rm CBM}$ in this set-up. In this subsection, we will give its proof. First of all, let us confirm that $(O_{\xi}(\alpha_0), d_{\rm CBM})$ is indeed a pseudo-metric space, where $\alpha_0$ is a fixed base point (as a contact 1-form of $\xi$). 

\begin{prop} \label{prop-cbm} The contact Banach-Mazur distance $d_{\rm CBM}$ on $O_{\xi}(\alpha_0)$ satisfies the following properties. For any $\alpha_1, \alpha_2, \alpha_3 \in O_{\xi}(\alpha_0)$,
\begin{itemize}
\item[(1)] $d_{\rm CBM}(\alpha_1, \alpha_2) \geq 0$, and $d_{\rm CBM}(\alpha_1, \alpha_1) = 0$;
\item[(2)] $d_{\rm CBM}(\alpha_1, \alpha_2) = d_{\rm CBM}(\alpha_2, \alpha_1)$;
\item[(3)] $d_{\rm CBM}(\alpha_1, \alpha_3) \leq d_{\rm CBM}(\alpha_1, \alpha_2) + d_{\rm CBM}(\alpha_2, \alpha_3)$; 
\item[(4)] $d_{\rm CBM}(\phi^*\alpha_1, \psi^*\alpha_2) = d_{\rm CBM}(\alpha_1, \alpha_2)$ for any $\phi, \psi \in {\rm Cont}_0(M, \xi)$;
\end{itemize}
\end{prop}

\begin{proof} The item (1) is trivial and the item (2) comes from the follow equivalence, 
\[ \frac{1}{C} \alpha_1 \preceq \phi^*\alpha_2 \preceq C \alpha_1 \Leftrightarrow \frac{1}{C} \alpha_2 \preceq (\phi^{-1})^*\alpha_1 \preceq C \alpha_2. \]
For (3), if $\frac{1}{C} \alpha_1 \preceq \phi^*\alpha_2 \preceq C \alpha_1$ and $\frac{1}{D} \alpha_2 \preceq \psi^*\alpha_3 \preceq D \alpha_2$ for some $C, D \geq 1$ and $\phi, \psi \in {\rm Cont}_0(M, \xi)$, then 
\[ \frac{1}{CD} \alpha_1 \preceq \frac{1}{D} \phi^*\alpha_2 \preceq \phi^* \psi^* \alpha_3 \preceq D \phi^*\alpha_2 \preceq (CD) \alpha_1. \]
For (4), it suffices to show $d_{\rm CBM}(\alpha_1, \psi^*\alpha_2) = d_{\rm CBM}(\alpha_1, \alpha_2)$. If $\frac{1}{C} \alpha_1 \preceq \phi^*\alpha_2 \preceq C \alpha_1$, then one has 
\[ \frac{1}{C} \alpha_1 \preceq (\phi^*(\psi^{-1})^*) (\psi^*\alpha_2) \preceq C \alpha_1\]
and vice versa. Thus we obtain the desired conclusion. 
\end{proof}

\begin{remark}\footnote{This remark is indebted to L. Polterovich.} The item (4) in Proposition \ref{prop-cbm} implies that $d_{\rm CBM}$ between contact 1-forms descends to a well-defined pseudo-metric, still denoted by $d_{\rm CBM}$, on the quotient space $O_{\xi}(\alpha_0)/{\rm Cont}_0(M, \xi)$.  One can readily check that this descended $d_{\rm CBM}$ is equal to the pseudo-distance defined and denoted by $d(a, b)$ on page 1528 in \cite{Pol02}. Since the contact mapping class group ${\rm Cont}(M,\xi)/{\rm Cont}_0(M, \xi)$ acts by isometries (of $d(a,b)$, hence also of $d_{\rm CBM}$) on $O_{\xi}(\alpha_0)/{\rm Cont}_0(M, \xi)$. This descended $d_{\rm CBM}$ can be used to study the dynamics of those contactomorphisms beyond ${\rm Cont}_0(M,\xi)$ as in, for instance, Proposition 3.1 in \cite{Pol02}. \end{remark}

Suppose that $(M, \xi = \ker\alpha_0)$ is closed and Liouville-fillable. Before we present the proof of Theorem \ref{thm-dc-lb}, here is an important observation. Any $\phi \in {\rm Cont}_0(M, \xi)$ can be lifted to an element $\Phi \in {\rm Symp}_0(\hat{W}, \omega)$, the identity component of the group of symplectomorphisms ${\rm Symp}(\hat{W}, \omega)$. To this end, we will take the approach from \cite{AM13}. Explicitly, consider any path $\{\phi_t\}_{t \in [0,1]}$ connecting $\mathds{1}$ and $\phi$ in ${\rm Cont}_0(M, \xi)$, and denote by $h(t,x):[0,1] \times M \to \R$ the associated contact Hamiltonian function with respect to any fixed $\alpha \in O_{\xi}(\alpha_0)$. Then its lifting to $[0,1] \times SM$ is defined by $H(t,u,x) = u \cdot h(t,x)$. For any (sufficiently large) $c>0$, consider a cut-off function $\beta_c \in C^{\infty}([0, \infty), [0,1])$ such that $\beta'_c\geq 0$, $\beta_c(u) = 1$ if $u \in [e^{-c}, \infty)$ and $\beta_c(u) = 0$ if $u \in [0, e^{-2c}]$. Modify the lift $H$ to be a smooth function $H_c:[0,1] \times \hat{W} \to \R$ defined by 
\begin{equation} \label{mod-H}
H_c|_{[0,1] \times (\hat{W} \backslash SM)} = 0, \,\,\,\,\mbox{and}\,\,\,\, H_c(t,u,x) = \beta_c(u)\cdot H(t,u,x).
\end{equation}
Denote by $\Phi_c$ the (symplectic) Hamiltonian diffeomorphism on $\hat{W}$ generated by $H_c$. Note that there always exists some (small) neighborhood of ${\rm Core}(\hat W)$ which is pointwisely fixed by $\Phi_c$, where the size of this neighborhood depends on $c$. Let us denote this neighborhood by $U_c(\hat{W})$, i.e., $U_c = \{(u,x) \in \hat{W} \,| \, u \in [0,e^{-2c}], x \in M\}$. 

For any Liouville domain $U \subset \hat{W}$ and $C>0$, recall that $CU$ is the image of $U$ under the flow of $L$ for time $\ln C$. Then we have the following lemma.  
\begin{lemma} \label{lemma-alpha}  The domain $W^{\alpha}$ constructed from (\ref{geo-alpha}) satisfies the following properties.
\begin{itemize}
\item[(1)] For $\alpha_1, \alpha_2 \in O_{\xi}(\alpha_0)$, if $\alpha_1 \preceq \alpha_2$, then $W^{\alpha_1} \subset W^{\alpha_2}$.
\item[(2)] For any $\alpha \in O_{\xi}(\alpha_0)$ and any $C \geq 0$, $W^{C \alpha} = C W^{\alpha}$. 
\item[(3)] For any $\alpha \in O_{\xi}(\alpha_0)$ and any $\phi \in {\rm Cont}_0(M, \xi)$, there exists a constant $c(\alpha)$, depending on $\alpha$, such that $W^{\phi^*\alpha} = \Phi_{c(\alpha)}^{-1}(W^{\alpha})$, where $\Phi_{c(\alpha)}$ is the (symplectic) Hamiltonian diffeomorphism generated by the modified Hamiltonian function defined in (\ref{mod-H}). 
\end{itemize}
\end{lemma}

\begin{proof} Suppose that $\alpha_1 = e^{f_1} \alpha_0$ and $\alpha_2 = e^{f_2} \alpha_0$. Then (1) and (2) directly come from the definition (\ref{geo-alpha}). For (3), suppose that $\phi^*\alpha_0 = e^{g_{\phi, \alpha_0}} \alpha_0$. Then $\phi^*\alpha = e^{f \circ \phi + g_{\phi, \alpha_0}} \alpha_0$ for some $f: M \to \R$. Then the definition (\ref{geo-alpha}) says that 
\[ W^{\phi^*\alpha} = \{(u',x') \in \hat{W} \,| \, u' < e^{f(\phi(x')) + g_{\phi, \alpha_0}(x')} \}. \]
Choose any path of contactomorphisms connecting $\mathds{1}$ and $\phi$, and denote by $h$ the associated contact Hamiltonian function with respect to $\alpha$. For any $c>0$ the Hamiltonian vector field $X_{H_c}$ of the modified lift Hamiltonian in (\ref{mod-H}) can be computed as follows. Split $X_{H_c} = X_L \frac{\partial}{\partial u} + X_M \frac{\partial}{\partial x}$ into two components. By definition, $dH_c = (\beta_c'(u) u + \beta_c(u)) h \cdot du + \beta_c(u) u \cdot dh$. By solving $\iota_X d(u\alpha) = - dH_c$, one obtain that 
\begin{equation} \label{ham-vf}
X_L = - \beta_c(u) u \cdot dh(R_\alpha) \,\,\,\, \mbox{and}\,\,\,\, X_M = (\beta_c'(u) u + \beta_c(u)) X_h + Y,
\end{equation}
where $X_h$ is the contact vector field on $M$ generated by $h$ and $Y$ is solved from the equation $d\alpha(Y) = \beta_c'(u) u \cdot (dh - dh(R_{\alpha}) \alpha)$. In particular, for $u \in [e^{-c}, \infty)$, $\beta_c(u)$ is constant $1$. Then $\beta_c'(u) = 0$, which implies that $Y = 0$, and then $X_L = - u \cdot dh(R_{\alpha})$ and $X_M = X_h$. Therefore, one easily obtain the following formula for $(u,x) \in \hat{W}$ where $u \in [e^{-c}, \infty)$, 
\begin{equation} \label{formular-Phi}
\Phi_c(u,x) = (e^{-g_{\phi, \alpha_0}(x)} u, \phi(x)). 
\end{equation}

For the star-shaped domain $W^{\phi^*\alpha}$, choose a sufficiently large $c(\alpha)>0$ (which will be specified later). For any point $(u',x') \in W^{\phi^*\alpha}$ where $u' \in [e^{-c(\alpha)}, \infty)$, by (\ref{formular-Phi}), $\Phi_{c(\alpha)}(u',x') =(e^{-g_{\phi, \alpha_0}(x')} u', \phi(x'))$. Then by the defining property of $W^{\phi^*\alpha}$, $e^{-g_{\phi, \alpha_0}(x')} u' < e^{f(\phi(x'))}$, which implies that $\Phi(u',x') \in W^{\alpha}$. On the other hand, for $u' \in [e^{-2c}, e^{-c}]$, the formula of $\Phi_{c(\alpha)}$ will not be as simple as the one in (\ref{formular-Phi}). However, by (\ref{ham-vf}), it is important to observe that the difference of $X_L$, compared with the previous case where $u' \in [e^{-c}, \infty)$, only involves with $\beta_c(u)$, which is uniformly bounded for any $x \in M$ and any $c >0$. Therefore, we can choose $c(\alpha)$ sufficiently large such that the ratio of the $u$-component of $\Phi_{c(\alpha)}(u',x')$ and $u'$ is sufficiently close to $1$ and also uniformly for any $x' \in M$. Explicitly, we will choose $c(\alpha)$ such that 
\[ \Phi_{c(\alpha)}(u',x') \in \left\{(u,x) \in \hat{W} \,\bigg| \, u < \min_{x \in M} e^{f(x)} \right\},\]
where $f$ is the function used to define $\alpha$, i.e., $\alpha = e^f \alpha_0$. Note that the minimum in the condition above exists due to the compactness of $M$. Then obviously $\Phi_{c(\alpha)}(u',x') \in W^{\alpha}$. Finally, for $u' \in [0, e^{-2c(\alpha)}]$, we can shrink $c(\alpha)$ further if it is necessary such that the $c(\alpha)$-dependent neighborhood $U_{c(\alpha)}$ of ${\rm Core}(\hat W)$ satisfies $U_{c(\alpha)} \subset W^{\alpha}$. Note that over $U_{c(\alpha)} $ the morphism $\Phi_{c(\alpha)}$ is just the identity, and then $\Phi_{c(\alpha)}(u',x') = (u',x') \in W^{\alpha}$ holds trivially. Thus, we have shown that for a sufficiently large $c(\alpha)$, $\Phi_{c(\alpha)}(W^{\phi^*\alpha}) \subset W^{\alpha}$. 

Apply the same argument to $\phi^{-1}$, then, for any contact 1-form $\beta$, we obtain that there exists a sufficiently large $c(\beta)$ such that $\Phi_{c(\beta)}^{-1} (W^{(\phi^{-1})^*\beta}) \subset W^{\beta}$. In particular, if $\beta = \phi^*\alpha$, we get that $\Phi_{c(\beta)}^{-1} (W^{(\phi^{-1})^*(\phi^*\alpha)}) = \Phi_{c(\beta)}^{-1} (W^{\alpha}) \subset W^{\phi^*\alpha}$, which is equivalent to $W^{\alpha} \subset \Phi_{c(\beta)}(W^{\phi^*\alpha})$. Replace $c(\alpha)$ by $c(\beta)$ if it is necessary, then we get the desired conclusion.\end{proof}

Now, we are ready to give the proof of Theorem \ref{thm-dc-lb}. 

\begin{proof} [Proof of Theorem \ref{thm-dc-lb}] Suppose there exists $\phi \in {\rm Cont}_0(M, \xi)$ and $C \geq 1$ such that $1/C \alpha_1 \preceq \phi^*\alpha_2 \preceq C \alpha_1$. By Lemma \ref{lemma-alpha}, there exists a sufficiently large $c(\alpha_2)$ such that the following relation holds, 
\[ \frac{1}{C} W^{\alpha_1} \subset \Phi_{c(\alpha_2)}^{-1}(W^{\alpha_2}) \subset C W^{\alpha_1}. \]
Then $\Phi_{c(\alpha_2)}|_{\frac{1}{C} W^{\alpha_1}}$ is a Liouville embedding from $\frac{1}{C} W^{\alpha_1}$ to $W^{\alpha_2}$, and $\Phi^{-1}_{c(\alpha_2)}|_{W^{\alpha_2}}$ is a Liouville embedding from $W^{\alpha_2}$ to $CW^{\alpha_1}$. Moreover, their composition is just the inclusion. Then, by the definition of $d_{\rm SBM}$, we know that $d_{\rm SBM}(W^{\alpha_1}, W^{\alpha_2}) \leq \ln C$. Thus we get the first conclusion. The second conclusion is trivial since by definition $d_c \leq d_{\rm SBM}$. 
\end{proof}

\subsection{Proof of Theorem \ref{thm-stability}} \label{ssect-proof-csi} \label{ssec-csi} The proof of Theorem \ref{thm-stability} relies on the contact shape invariant, which was rigorously defined in \cite{Eli91} for the first time (and see also \cite{Mul19}). Let us recall its definition first. Given a domain $U \times S^1$ of $W \times S^1$, and denote by $\tau$ the identity map in the automorphism group ${\rm Aut}(H^1(X \times S^1; \R))$. Consider the symplectization $U \times S^1 \times \R_+ \subset W \times S^1 \times \R_+$, where $r$ denotes the coordinate of $\R_+$. Note that by definition the symplectic structure on $W \times S^1 \times \R_+$ (hence, also on $U \times S^1 \times \R_+$) is $d(r (\lambda + dt))$. Fix any contact hypersurface $M$ of $W$, e.g., a unit cosphere bundle of $X$ with respect to a certain metric on $X$, and the coordinates $(s,x,t,r)$ denote a point in $W \times S^1 \times \R_+$, where $x \in M$. Consider the diffeomorphism $\Phi$ on $W \times S^1 \times  \R_+$ defined by $\Phi((s,x,t,r)) = (sr,x,t,r)$ for any point $(s,x,t,r) \in W \times S^1 \times \R_+$. It is easy to check that $\Phi$ is a symplectomorphism from $(W \times S^1 \times  \R_+, d(r(\lambda + dt)))$ to $(W \times S^1 \times  \R_+, \omega + dr \wedge dt)$. Therefore, the image $\Phi(U \times S^1 \times \R_+)$ can be regarded as a domain in the standard (symplectic) stabilization of $(W, \omega)$, i.e., $W \times T^*S^1$, with respect to the standard split symplectic structure. Now, consider the set defined as 
\begin{equation} \label{dfn-ssi}
{\rm sh}(\Phi(U \times S^1 \times \R_+); \tau) : = \left\{(a, b) \in H^1(X \times S^1; \R) \,\Bigg| \, \begin{array}{l} \mbox{there exists a Lagrangian emb.} \\ \mbox{$f: X \times S^1 \to \Phi(U \times S^1 \times \R_+)$} \\ \mbox{such that $f^*|_{H^1(X \times S^1; \R)} = \tau$} \\ \mbox{and also $f^*(\lambda + dt) = (a, b)$}\end{array} \right\}.
\end{equation}

A few remarks about the notation in the definition (\ref{dfn-ssi}) are in order. First, since $H^1(X \times S^1; \R) = H^1(X; \R) \oplus H^1(S^1; \R)$, the vector $(a, b)$ is understood as $a \in H^1(X; \R)$ and $b \in H^1(S^1; \R)$. Moreover, the equality $f^*(\lambda + dt) = (a, b)$ is understood as the closed 1-form $f^*(\lambda + dt)$ evaluated at a chosen basis of $H_1(X \times S^1; \Z)$. Second, since $\Phi$ does not change the homotopy type, $f^* = (\Phi^{-1} \circ f)^*$. Moreover, an existence of a Lagrangian embedding $f: X \times S^1 \to \Phi(U \times S^1 \times \R_+)$ is equivalent to an existence of a Lagrangian embedding $\Phi^{-1} \circ f: X \times S^1 \to U \times S^1 \times \R_+$, since $\Phi$ is symplectic. Third, since $U$ is a star-shaped domain of $W = T^*X$, the domain $U \times S^1 \times \R_+$ deformation retracts to $X \times S^1 \times \R_+$ and $H^1(U \times S^1 \times \R_+; \R) \simeq H^1(X \times S^1; \R)$, which guarantees that the condition $f^*|_{H^1(X \times S^1; \R)} = \tau$ in (\ref{dfn-ssi}) is well-defined.  

An important observation is that the set ${\rm sh}(\Phi(U \times S^1 \times \R_+); \tau)$ admits a diagonal $\R_+$-action, i.e., if $(a, b) \in {\rm sh}(\Phi(U \times S^1 \times \R_+); \tau)$, then for any fixed $r_* \in \R_+$, $(r_* a, r_* b) \in {\rm sh}(\Phi(U \times S^1 \times \R_+); \tau)$. In fact, for this $r_*$, denote by $r_* \circ$ the diagonal action on $W \times S^1 \times \R_+$ defined by $r_*\circ(s,x,t, r) : = (sr_*, x, t, rr_*)$. Then if $f: X \times S^1 \to \Phi(U \times S^1 \times \R_+)$ is a Lagrangian embedding, then 
\[ f_{r_*} : = \Phi \circ r_* \circ \Phi^{-1} \circ f: X \times S^1 \to \Phi(U \times S^1 \times \R_+) \]
is also a Lagrangian embedding. Moreover, if $f^*(\lambda + dt)  =(a, b)$, then $f_{r_*}^*(\lambda + dt) = (r_* a, r_*b)$. 

Denote by ${\rm pr}_1: H^1(X \times S^1; \R) \to H^1(X; \R)$ the natural projection. The {\it contact shape invariant of $U \times S^1$ with respect to $\tau$} is defined as 
\begin{equation} \label{dfn-csi}
{\rm csh}(U \times S^1; \tau): = {\rm pr}_1({\rm sh}(\Phi(U \times S^1 \times \R_+); \tau)/ {\mbox{diagonal $\R_+$-action}}). 
\end{equation}
As a subset of $\R^m$ where $m = \dim_{\R} H^1(X; \R)$, the contact shape invariant ${\rm csh}(\Phi(U \times S^1 \times \R_+); \tau)$ can be equivalently obtained by projecting the subset ${\rm sh}(\Phi(U \times S^1 \times \R_+); \tau) \cap \{b =1\}$ via ${\rm pr}_1$, where $b$ represents the coordinate of $\R = H^1(S^1; \R)$.

\begin{lemma} \label{csh-functor} The contact shape invariant satisfies the following functorial properties. 
\begin{itemize}
\item[(1)] If $U$ is a star-shaped domain of $W$, then ${\rm csh}(CU \times S^1; \tau) = C\cdot {\rm csh}(U \times S^1; \tau)$ for any $C >0$, where $C\cdot$ is the standard rescaling in a Euclidean space. \footnote{It is easy to see, by Lagrangian Weinstein neighborhood theorem, ${\rm csh}(CU \times S^1; \tau)$ is always an open subset of $\R^m$, and $0 \in \R^m$ is contained inside its closure.}
\item[(2)] Let $U,V$ be star-shaped domains of $W \times S^1$. If there exists a contact isotopy $\phi = \{\phi_t\}_{t \in [0,1]}$ such that $\phi_1(U \times S^1) \subset V \times S^1$, then ${\rm csh}(U \times S^1; \tau) \subset {\rm csh}(V \times S^1; \tau)$.
\item[(3)] If $U$ is a star-shaped domain of $W$, then for any $k \in \N$, 
\begin{equation} \label{k-csh}
{\rm csh}(U \times S^1; \tau) \subset k \cdot {\rm csh}((U \times S^1)/k; \tau), 
\end{equation}
where $(U \times S^1)/k = \tau_k^{-1}(U \times S^1)$ and $\tau_k: W \times S^1 \to W \times S^1$ is the covering map defined in (\ref{cover}). 
\end{itemize}
\end{lemma}

\begin{proof} The proof of (1) and (2) are straightforward, so we only prove (3) here. By definition, $\tau_k^{*}(\lambda + dt) = k \cdot (\lambda + dt)$. Therefore, $(U \times S^1)/k$ is a domain of $(W \times S^1 \times \R_+, k (d\lambda + dr \wedge dt))$, where the symplectic form is changed by the factor $k$. If there exists a Lagrangian embedding $f: X \times S^1 \to U \times S^1 \times \R_+$, then it lifts to a Lagrangian embedding $\tilde{f}^{(k)}: X \times S^1 \to (U \times S^1)/k \times \R_+$. Moreover, if $f^*(\lambda + dt) = (a, b)$, then $(\tilde{f}^{(k)})^*(k\lambda + k dt) = (a, kb)$. Note that the factor $a$ does not change since by definition the $U$-component of the embedding $\tilde{f}^{(k)}$ is rescaled by $\frac{1}{k}$. This implies that if $(a,b) \in {\rm sh}(\Phi(U \times S^1 \times \R_+); \tau)$, then $(a, kb) \in {\rm sh}(\Phi((U \times S^1)/k \times \R_+); \tau)$. Then, modulo the diagonal $\R_+$-action, we get the desired conclusion. \footnote{It can be confusing for the factor in item (3) in Lemma \ref{csh-functor} - whether it is $k$ or $\frac{1}{k}$. From our proof, if $(a,1) \in {\rm sh}(\Phi(U \times S^1 \times \R_+); \tau)$, then $(a, k) \in {\rm sh}(\Phi((U \times S^1)/k \times \R_+); \tau)$. Therefore, in order to obtain ${\rm csh}((U \times S^1)/k; \tau)$ via intersecting $``\{b=1\}"$, we need to collect $\frac{1}{k}a$. In other words, $\frac{1}{k} {\rm csh}(U \times S^1; \tau) \subset {\rm csh}((U \times S^1)/k; \tau)$. This is the desired conclusion in (\ref{k-csh}) in Lemma \ref{csh-functor}.}\end{proof}

In general, we only get the inclusion in (\ref{k-csh}), since a Lagrangian embedding always admits a lift, but not every Lagrangian embedding $f: X \times S^1 \to  (U \times S^1)/k \times \R_+$ can be pushed down via $\tau_k$ to be a well-defined Lagrangian embedding from $X \times S^1 \to U \times S^1 \times \R_+$. In order to do so, we require some periodicity condition of the embedding $f$. However, in the following example, we will examine a special case where (\ref{k-csh}) is indeed an equality. 

\begin{ex} \label{ex-torus} Let $X = \T^n$ (so $W = T^*\T^n$), and $U = \T^n \times A_U$ for some open subset $A_U$ of $\R^n$ containing $0 \in \R^n$. Then we claim that 
\begin{equation} \label{csh-torus}
{\rm csh}(U \times S^1; \tau) = k \cdot {\rm csh}((U \times S^1)/k; \tau) = A_U. 
\end{equation}
We will show that ${\rm csh}(U \times S^1; \tau) = A_U$ and ${\rm csh}((U \times S^1)/k; \tau) = \frac{1}{k} \cdot A_U$. By definition, 
\begin{align*}
\Phi(U \times S^1 \times \R_+) & = \Phi(\T^n \times A_U \times S^1 \times \R^+) \\
& = \{(q, rp, t, r) \in W \times S^1 \times \R_+ \,| \, (q,p,t,r) \in \T^n \times A_U \times S^1 \times \R_+ \}\\
& = \{(q, t, rp, r) \in \T^{n+1} \times (\R^n \times \R_+) \,| \, p \in A_U, r \in \R_+\}
\end{align*}
By Sikorav's theorem (see \cite{Sik89} or \cite{Eli91}) (which eventually comes from Gromov's famous theorem on the exact Lagrangian embeddings of cotangent bundles in \cite{Gro85}), we know that 
\[ {\rm sh}(\Phi(U \times S^1 \times \R_+); \tau) = \{(rp, r) \in \R^n \times \R_+ \,| \, p \in A_U, r\in \R_+\}.\]
Therefore, ${\rm csh}(U \times S^1; \tau) = A_U$.

The same argument shows that 
\begin{align*}
{\rm sh}(\Phi((U \times S^1)/k \times \R_+); \tau) & = \left\{(k r p, kr) \in \R^n \times \R_+ \,\bigg| \, p \in \frac{1}{k}\cdot A_U, r \in \R_+\right\}\\
&= \{(rp, kr) \in \R^n \times \R_+ \,| \, p \in A_U, r\in \R_+\}.
\end{align*}
Therefore, ${\rm csh}(U \times S^1; \tau) = \frac{1}{k} \cdot A_U$.
\end{ex}

Now, we are ready to give the proofs of Theorem \ref{thm-stability} and Corollary \ref{cor-large-scale-torus}.
 
\begin{proof} [Proof of Theorem \ref{thm-stability}] By definition, for any $\ep>0$, there exists a pair $(k,l) \in \N^2$ such that $\frac{k}{l} \leq \rho_c(U \times S^1, V \times S^1) + \ep$. In other words, there exists a contact isotopy of $W \times S^1$, denoted by $\phi = \{\phi_t\}_{t \in [0,1]}$, such that $\phi_1((U \times S^1)/k) \subset (V \times S^1)/l$. By (2) in Lemma \ref{csh-functor}, ${\rm csh}((U \times S^1)/k; \tau) \subset {\rm csh}((V \times S^1)/l; \tau)$. Moreover, by (\ref{csh-torus}) in Example \ref{ex-torus}, we know that 
\[ A_U/k \subset A_V/l, \,\,\,\,\mbox{which implies} \,\,\,\, A_U \subset \frac{k}{l} A_V. \]
Note that $\delta(\cdot, \cdot)$ is symmetric, and it can be written as $\delta(A,B) = \max\{\delta_-(A,B), \delta_+(A,B)\}$, where $\delta_-(A,B) = \inf\{C>1\,| \, A \subset C\cdot B\}$ and $\delta_+(A,B) = \inf\{C>1 \,| \, B \subset C\cdot A\}$. Therefore, $\delta_-(A_U, A_V) \leq \frac{k}{l} \leq \rho_c(U \times S^1, V \times S^1) + \ep$.  Then a symmetric argument leads to the desired conclusion. \end{proof} 

\begin{proof}[Proof of Corollary \ref{cor-large-scale-torus}]  Given $k \in \N$, consider a morphism $\Psi_k: \R^k \to \mathcal S(T^*\T^n)$ defined along the following few steps. First, observe that there exists a quasi-isometric embedding $\mathbb L: \R \to [0, \infty)^2$ by an ``L-shape''. Explicitly, for $x<0$, $\mathbb L(x) = (1, -x+1)$, and for $x \geq 0$, $\mathbb L(x) = (1+x, 1)$. In general, for $x = (x_1, ..., x_k) \in \R^k$, for brevity still use $\mathbb L$ to denote the map $\mathbb L: \R^k \to [0, \infty)^{2k}$ defined by $\mathbb L(x) = (\mathbb L(x_1), ..., \mathbb L(x_k))$. It is easy to check that 
\begin{equation} \label{L-est}
\frac{1}{2} |x-y|_{\infty} \leq |\mathbb L(x) - \mathbb L(y)|_{\infty} \leq |x-y|_{\infty}. 
\end{equation}
Therefore, we are reduced to construct a quasi-isometric embedding $\Psi_k$ from the metric space $([0, \infty)^{2k}, |\cdot|_{\infty})$ to $(S(T^*\T^n),d_c)$. Second, since $n \geq 2$, $H^1(\T^n; \R)$ contains a 2-dimensional Euclidean space $\R^2$, say, in coordinates $x_1$ and $x_2$. In this $\R^2$, choose $2k$-many non-collinear directions represented by the unit vectors $e_1, ..., e_{2k}$. Choose a sufficiently large constant $C_0$, and denote by $\ell_i$ the line segment $\{t e_i\,| \, t \in [0,C_0]\}$ for each $i \in \{1, ..., 2k\}$. The 1-skeleton of the $x_1x_2$-plane formed by the union of the line segments $\{\ell_i\}_{i=1,..., 2k}$ is denoted by $X_{\rm base}$. For any $v = (v_1, ..., v_{2k}) \in [0, \infty)^{2k}$, the notation $v \cdot X_{\rm base}$ means a new 1-skeleton where $\ell_i$ is rescaled to $e^{v_i} \ell_i$ for each $i \in\{1, ..., 2k\}$. Moreover, for the rest $(n-2)$-dimensional Euclidean subspace $\R^{n-2} \subset H^1(\T^n; \R)$, consider the open higher dimensional unit cube $[0,1)^{n-2} \subset \R^{n-2}$. Third, choose an appropriate width $\ep(v)$ of the 1-skeleton $v \cdot X_{\rm base}$ in the $x_1x_2$-plane such that the resulting open region denoted by $\Omega_{v}$ satisfies  
\begin{equation} \label{normalization}
{\rm vol}_{\R^n}(\Omega_{v} \times [0,1)^{n-2}) = 1 / {\rm vol}_g (\T^n)
\end{equation}
for a fixed Riemannian metric $g$ on the base $\T^n$. Denote $A_{v} : = \Omega_{v} \times [0,1)^{n-2}$. Define $\Psi_k: [0, \infty)^{2k} \to \mathcal S(T^*\T^n)$ by $\Psi_k(v) = \T^n \times A_{v}$. 

Now, let us confirm that $\Psi_k$ is the desired quasi-isometric embedding. First, for two vectors $v, w \in [0, \infty)^{2k}$, by Proposition \ref{prop-cbm-c} and Theorem \ref{thm-stability}, 
\[ \delta(A_{v}, A_{w}) \leq d_c(\T^n \times A_{v}, \T^n \times A_{w}) = d_c(\Psi_k(v), \Psi_k(w)).\]
By our construction, $\delta(A_{v}, A_{w}) = |v - w|_{\infty}$ since each direction $\ell_i$ must be rescaled by $|v_i - w_i|$ in order to satisfy the inclusion relations in the definition of $\delta(\cdot, \cdot)$. Therefore, $|v - w|_{\infty} \leq d_c(\Psi_k(v), \Psi_k(w))$. On the other hand, for a sufficiently large $C_0$, the width $\ep(v)$ is sufficiently small for every $v \in [0, \infty)^{2k}$. Then, the volume ${\rm vol}_{\R^n}(\Omega_{v} \times [0,1)^{n-2})$ can be computed as 
\begin{equation} \label{vol}
{\rm vol}_{\R^n}(\Omega_{v} \times [0,1)^{n-2}) = \ep(v)C_0 \cdot \sum_{i=1}^{2k} e^{v_i} + o(\ep(v)). 
\end{equation}
For $v, w \in [0, \infty)^{2k}$, without loss of generality, assume that $|v - w|_{\infty} = w_1 - v_1 \geq 0$. In particular, for every $i \in \{1, ..., 2k\}$, $w_i - v_i \leq |v_i - w_i| \leq w_1 - v_1$. For $\Psi_k(v)$ and $\Psi_k(w)$, we claim that choosing the constant $C : = C_1 e^{w_1 - v_1}$ for a constant $C_1 >1$ yields $\Psi_k(v) \subset C\Psi_k(w)$. Then a symmetric argument implies the desired conclusion. 

To prove the claim, for brevity, let us neglect the approximation term $o(\ep(v))$ in (\ref{vol}), since this can be adjusted by a proper constant $C_1>1$. From (\ref{vol}) and (\ref{normalization}), we know that $\ep(v) \cdot \sum_{i=1}^{2k} e^{v_i} = \ep(w) \cdot \sum_{i=1}^{2k} e^{w_i}$. Observe that $e^{w_1 - v_1} \geq \frac{\sum_{i=1}^{2k} e^{w_i}}{\sum_{i=1}^{2k} e^{v_i}}$. \footnote{In fact, denote by $a_i = e^{w_i}$ and $b_i = e^{v_i}$. By our hypothesis, $\frac{a_i}{b_i} = e^{w_i - v_i} \leq e^{w_1 - v_1} = \frac{a_1}{b_1}$ for each $i \in \{1, ..., 2k\}$. Set $M = \frac{a_1}{b_1}$, then $a_i \leq M b_i$ for each $i \in \{1, ..., 2k\}$. Then sum over all such $i$, and we get $\sum_{i=1}^{2k} a_i \leq M \sum_{i=1}^{2k} b_i$, which implies that $\frac{\sum_{i=1}^{2k} a_i}{\sum_{i=1}^{2k} b_i} \leq M$.} Then 
\begin{equation*}
e^{w_1 - v_1} \cdot \ep(w) = e^{w_1 -v_1} \cdot \frac{\sum_{i=1}^{2k} e^{v_i}}{\sum_{i=1}^{2k} e^{w_i}} \ep(v) \geq \ep(v).
\end{equation*}
This implies that $A_{v} \subset C A_{w}$, and thus we get the claim. \end{proof}

\subsection{Proof of Theorem \ref{rgr-cbm} }\label{subsec-pf-rgr-cbm}

\begin{proof} [Proof of Theorem \ref{rgr-cbm}] It suffices to prove that 
\begin{equation} \label{rho-comp}
\rho^-_{\geq_+}([\phi], [\psi]) \geq \rho_c(U([\phi]), U([\psi])),
\end{equation}
where $\rho^-_{\geq_+}(a,b) = \rho^+_{\geq_+}(b,a)$ for any inputs $a$ and $b$, and $\rho^+_{\geq_+}$ and $\rho_c$ are defined in (\ref{dfn-rgr}) and (\ref{rho-c}), respectively. Then the desired conclusion comes from a symmetric argument. For any $\ep>0$, there exists a pair $(k,l) \in \N^2$ such that $[\psi]^k = [\psi^k] \geq_+ [\phi]^l = [\phi^l]$, and $\frac{k}{l} \leq \rho^-_{\geq_+}([\phi], [\psi]) + \ep$. Then, by the definition of the partial order $\geq_+$, there exist a positive contact isotopy $\psi'$, homotopic to $\psi^k$, and a positive contact isotopy $\phi'$, homotopic to $\phi^l$, such that $\psi' \geq_+ \phi'$. For any fixed contact 1-form $\alpha$, the corresponding contact Hamiltonian functions satisfy $h_{\alpha}(\psi') \geq h_{\alpha}(\phi')$. Passing to the star-shaped domains defined in (\ref{geo-ss}), we know that, for any $s>0$, 
\begin{equation*}
\left\{(s,x,t) \in SM \times S^1 \,\bigg| \, h_{\alpha}(\phi') < \frac{1}{s} \right\} \subset \left\{(s,x,t) \in SM \times S^1 \,\bigg| \, h_{\alpha}(\psi') < \frac{1}{s} \right\},
\end{equation*}
which yields that $U(\phi') \subset U(\psi')$. Meanwhile, there exist ambient contact isotopies on $\hat{W} \times S^1$ denoted by $\Phi = \{\Phi_t\}_{t \in [0,1]}$ and $\Psi = \{\Psi_t\}_{t \in [0,1]}$ such that $U(\phi') = \Phi_1(U(\phi^k))$ and $U(\psi') = \Psi_1(U(\psi^l))$. Then the contact isotopy $\Phi^{-1} \circ \Psi$ satisfies 
\begin{equation} \label{inclusion}
(\Psi^{-1} \circ \Phi)_1(U(\phi^k)) = \Psi_1^{-1} (\Phi_1(U(\phi^k))) = \Psi_1^{-1}(U(\phi')) \subset U(\psi^l).
\end{equation} 
Moreover, assume that both $h_{\alpha}(\phi)$ and $h_{\alpha}(\psi)$ are 1-periodic. Then $\phi^k$ and $\psi^l$ can be generated by $k\cdot h_{\alpha}(x,kt)$ and $l \cdot h_{\alpha}(x, lt)$, respectively. It is important to observe that 
\[ U(\phi^k) = \tau_k^{-1} (U(\phi)) = U(\phi)/k \,\,\,\,\mbox{and}\,\,\,\, U(\psi^l) = \tau_l^{-1} (U(\psi)) = U(\psi)/l, \]
where $\tau_k$ and $\tau_l$ are the covering maps defined in (\ref{cover}).  Therefore, we know that 
\[ \rho_c(U(\phi), U(\psi)) \leq \frac{k}{l} \leq \rho^-_{\geq_+}([\phi], [\psi]) + \ep. \]
Letting $\ep$ go to zero, together with Theorem \ref{lemma-ccbm-property}, we get the desired conclusion. 
\end{proof}

\subsection{Proof of Theorem \ref{prop-sheaf-order}} \label{subsec-pf-prop-sheaf-order} In preparation for our discussion, let us recall some standard notations. The sheaf convolution denoted by $``\circ|_I"$ (see (1.13) in \cite{GKS12}) provides a well-defined bi-operator on $\D^b(\k_{X \times X \times I})$. Explicitly, consider the projection $p_{ijI}: X_1 \times X_2 \times X_3 \times I \to X_i \times X_j \times I$, where $X_i = X$ for $i=1,2,3$. The sheaf convolution is defined by
\[ \F \circ|_I \G : = {\rm R}{p_{13I}}_! (p_{12I}^{-1} \F \otimes^L p_{23I}^{-1} \G), \]
for any two elements $\F, \G \in \D^b(\k_{X \times X \times I})$, where $\otimes^L$ means the external product. Under this sheave convolution, the identity element in $\D^b(\k_{X \times X \times I})$ is $\k_{\Delta \times I}$ where $\Delta \subset X \times X$ is the diagonal. When $I$ is empty, we simply write the corresponding sheaf convolution as $\F \circ \G$, which is defined as above without the $I$-components. For an element $\F \in \D^b(\k_{X \times X \times I})$, its ``formal inverse'' (with respect to the sheaf convolution) denoted by $\F^{-1} \in \D^b(\k_{X \times X \times I})$ is defined by (see (1.21) or page 216 in \cite{GKS12})
\begin{equation} \label{sheaf-inverse}
\F^{-1} : = (v \times \mathds{1}_I)^{-1} {\rm R} \mathcal{H}om(\F, \omega_M \boxtimes \k_M \boxtimes \k_I),
\end{equation}
where $v: X \times X \to X \times X$ is the swap map, i.e., $v(x_1,x_2) = (x_2, x_1)$ for any $(x_1,x_2) \in X \times X$, and $\omega_X \in \D^b(\k_X)$ is the dualizing complex on $M$, which is isomorphic to the orientation sheaf shifted by the dimension (see Definition 3.1.16 in \cite{KS90}). We call $\F^{-1}$ the ``formal inverse'' since certain conditions on $\F$ are necessary in order to have $\F \circ|_I \F^{-1} = \F^{-1} \circ|_I \F \simeq \k_{\Delta \times I}$ (see Proposition 1.14 in \cite{GKS12}). Any such $\F$ which also satisfies the normalization condition that $\F|_{t=0} = \k_{\Delta}$ is called {\it admissible}, and the set of all admissible $\F \in \D^b(\k_{X \times X \times I})$ is denoted by $\D_{\rm adm}^b(\k_{X \times X \times I})$. In Example \ref{ex-sq} below, we will see that $\D_{\rm adm}^b(\k_{X \times X \times I})$ in fact contains interesting elements. Finally, for any $k \in\N$, recall that  $\F^k : = \F \circ|_I \cdots \circ|_I \F$, the sheaf convolution of $k$ many $\F$.

Another important ingredient is the singular support of an element $\F \in \D^b(\k_{X \times X \times I})$, denoted by $SS(\F)$. It is a conical subset of $T^*(X \times X \times I) = T^*X \times T^*X \times T^*I$. The explicit definition of $SS$ will not be necessary here, and we refer to Chapter V in \cite{KS90} for a systematic study of $SS$. Recall that for two subset $\Lambda_1, \Lambda_2 \subset T^*(X \times X \times I)$, the {\it correspondence} $\Lambda_1 \circ|_I \Lambda_2$, which is a subset of $T^*(X \times X \times I)$, is defined by 
{\small{\begin{equation} \label{corres}
\Lambda_1 \circ|_I \Lambda_2 : = \left\{((x_1, \xi_1), (x_3, \xi_3), (t, \tau))\, \bigg| \, \exists (x_2, \xi_2), (t, \tau_1), (t, \tau_2) \,\,\mbox{s.t.} \,\, \begin{array}{l} ((x_1, \xi_1), (x_2, - \xi_2), (t, \tau_1)) \in \Lambda_1 \\ ((x_2, \xi_2), (x_3, \xi_3), (t, \tau_2)) \in \Lambda_2 \\ \tau = \tau_1 + \tau_2 \end{array} \right\}.
\end{equation}}}

The following two properties of $SS$ are particularly useful to us. Proposition \ref{prop-ss-1} is an important proposition for $SS$, since it reflects the ``geometry'' of $SS$. Proposition \ref{prop-ss-2} shows that the singular support of a sheaf can put strong restrictions on the behavior of the sheaf itself. The proof of Proposition \ref{prop-ss-2} is in fact quite difficult, where the main step is based on the microlocal Morse lemma (see Corollary 5.4.19 in \cite{KS90}). 

\begin{prop} \label{prop-ss-1} For any $\F, \G \in \D^b(\k_{X \times X \times I})$, under a certain condition ((1.14) in \cite{GKS12}), $SS(\F \circ|_I \G) \subset SS(\F) \circ|_I SS(\G)$, where the $\circ|_I$ on the right hand side is defined in (\ref{corres}). \end{prop}

\begin{prop} \label{prop-ss-2} Given $\F \in \D^b(\k_{X \times X \times I})$, assume that $SS(\F) \subset T^*(X \times X) \times 0_I$, then $\F \simeq p^{-1} {\rm R}p_*\F$, where $p: X \times X \times I \to X \times X$ is the projection. \end{prop}

The following example summarizes the main result from \cite{GKS12}, which serves as a supporting example, with the origin in contact geometry, of the elements in $\D^b(\k_{X \times X \times I})$.

\begin{ex}[Sheaf quantizations \cite{GKS12}]\label{ex-sq} This example will demonstrate a family of elements in $ \D_{\rm adm}^b(\k_{X \times X \times I})$ which comes from contact geometry. Let $X$ be a closed manifold and $g$ be a metric on $X$. Denote by $M = S_g^*X$ the unit co-sphere bundle, and it admits a contact structure $\xi$ which is naturally induced from the canonical symplectic structure on $T^*X$. We have seen above that a contact isotopy $\phi = \{\phi_t\}_{t \in I}$ lifts to a 1-homogeneous Hamiltonian isotopy $\Phi = \{\Phi_t\}_{t \in I}$ on the symplectization $SM = \dot T^*M := T^*M \backslash 0_M$. The main result in \cite{GKS12} says that for any such $\Phi: I \times \dot T^*M \to \dot T^*M$, there exists a unique 
element $\K_{\phi} \in D_{\rm adm}^b(\k_{X \times X \times I})$, called the {\it sheaf quantization of $\phi$}, such that $SS(\K_{\phi})$ satisfies 
\begin{align*}
SS(\K_{\phi}) & \subset \Lambda_{\Phi} \cup 0_{X \times X \times I} \\
& = \{((x,\xi), - \Phi_t(x,\xi), (t, - H_t(\Phi_t(x,\xi)))) \,|\, x \in X, \, t \in I\} \cup 0_{X \times X \times I}, 
\end{align*}
where $\Lambda_{\Phi}$ is called the Lagrangian suspension of the (symplectic) Hamiltonian isotopy $\Phi$, and $H_t$ is the (symplectic) Hamiltonian function of $\Phi$. More explicitly, assume that the contact isotopy $\phi$ is generated by a contact Hamiltonian function $h_t: I \times M \to \R$, then $H_t = s \cdot h_t$, where $s$ denotes the $\R_+$-component of the symplectization of $M$. Observe that $\K_{\phi}$ is a dominant if and only if the contact Hamiltonian function $h_t$ of $\phi$ is strictly positive. Moreover, we have the following two results.
\begin{itemize}
\item[(i)] By the uniqueness of the main result in \cite{GKS12}, $(\K_{\phi})^{-1}  \simeq \K_{\phi^{-1}}$.
\item[(ii)] For any two contact isotopies $\phi, \psi$ on $(M, \xi)$, $\K_{\phi} \circ|_I \K_{\psi} = \K_{\phi \circ \psi}$. 
\end{itemize}
As a consequence, we know that $\K_{\phi} \geq_s \K_{\psi}$ if and only if $\K_{\phi^{-1} \circ \psi}$ satisfies the condition in Definition \ref{sheaf-order}, i.e., $SS(\K_{\phi^{-1} \circ \psi}) \subset T_{\{\tau \leq 0\}} (X \times X \times I)$. 
\end{ex}

Now, we are ready to give the proof of Theorem \ref{prop-sheaf-order}.

\begin{proof} [Proof of Theorem \ref{prop-sheaf-order}] Let us prove the first conclusion. By the admissibility hypothesis, for any $\F \in \D_{\rm adm}^b(\k_{X \times X \times I})$, $\F^{-1} \circ|_I \F = \k_{\Delta \times I}$. Since $SS(\k_{\Delta \times I}) \subset (\Delta_{T^*X} \times 0_I) \cup 0_{X \times X \times I}$, which is certainly contained in $T^*_{\{\tau \leq 0\}} (X \times X \times I)$, we know $\F \geq_s \F$, and we get the reflexivity of the relation $\geq_s$. Next, if $\F \geq_s \G$, and $\G \geq_s \H$, then by associativity of the operator $\circ|_I$, we know $SS(\F^{-1} \circ|_I \H) = SS((\F^{-1} \circ|_I \G) \circ|_I (\G^{-1} \circ|_I \H))$. By Proposition \ref{prop-ss-1}, 
\[ SS((\F^{-1} \circ|_I \G) \circ|_I (\G^{-1} \circ|_I \H)) \subset SS(\F^{-1} \circ|_I \G) \circ|_I SS(\G^{-1} \circ|_I \H). \]
By the explicit formula (\ref{corres}), the $\tau$-components of $SS(\F^{-1} \circ|_I \H)$ come from certain sums of $\tau_1$-components of $SS(\F^{-1} \circ|_I \G)$ and $\tau_2$-components of $SS(\G^{-1} \circ|_I \H)$. Our hypothesis implies that both $\tau_1$-components and $\tau_2$-components are non-positive, which yields that $\tau$-components are also non-positive. This proves the transitivity of the relation $\geq$. Next, if $\F \geq_s \G$ and $\G \geq_s \F$, then it is readily to check that $SS(\F^{-1} \circ|_I \G) \subset T^*(X \times X) \times 0_I$. By Proposition \ref{prop-ss-2}, $\F^{-1} \circ|_I \G \simeq p^{-1} {\rm R}p_*(\F^{-1} \circ|_I \G)$ where $p: X \times X \times I \to X \times X$ is the projection. For every $t \in I$, denote by $\iota_t: X \times X \times \{t\} \to X \times X \times I$ the inclusion. Then for every $t \in I$, 
\begin{align*}
(\F^{-1} \circ|_I \G)|_t & \simeq \iota_t^{-1} (p^{-1} {\rm R}p_*(\F^{-1} \circ|_I \G)) \\
& = (p \circ \iota_t)^{-1} {\rm R}p_*(\F^{-1} \circ|_I \G)\\
& = {\rm R}p_*(\F^{-1} \circ|_I \G) \\
& = (p \circ \iota_0)^{-1} {\rm R}p_*(\F^{-1} \circ|_I \G)\\
& = \iota_0^{-1} (p^{-1} {\rm R}p_*(\F^{-1} \circ|_I \G)) \\
& \simeq (\F^{-1} \circ|_I \G)|_0 = \F^{-1}|_{t=0} \circ \G|_{t=0} = \k_{\Delta}.
\end{align*}
Therefore, $\F^{-1} \circ|_I \G \simeq \k_{\Delta \times I}$. Convoluting $\F$ on both sides, we get $\G \simeq \F$. Thus, we obtain the anti-symmetry of the relation $\geq_s$. Therefore, $\geq_s$ is a partial order.

Let us prove the second conclusion. Let $M = S_g^*X$. Consider the sheaf quantization demonstrated in Example \ref{ex-sq}. It provides a map denoted by $\sigma: {\rm PCont}_0(M,\xi) \to \D_{\rm adm}^b(\k_{X \times X \times I})$ given by $\phi \to \mathcal K_{\phi}$.  In particular, it restricts to a map $\sigma: {\rm PCont}_+(M, \xi) \to \D_{{\rm adm} \cap {\rm dom}}^b(\k_{X \times X \times I}) = \D_{\rm adm}^b(\k_{X \times X \times I})\cap \D_{\rm dom}^b(\k_{X \times X \times I})$. By the conclusion at the end of Example \ref{ex-sq} (which is derived from (i) and (ii) in Example \ref{ex-sq}), the relative growth rates are related by $\gamma_{\geq_o}(\phi, \psi) \geq \gamma_{\geq_+}(\mathcal K_{\phi}, \mathcal K_{\psi})$. Therefore, passing to the pseudo-metrics $d_{\geq_o}$ and $d_{\geq_s}$ constructed from $\gamma_{\geq_o}$ and $\gamma_{\geq_s}$ as in (\ref{dfn-pm}), respectively, for any $\phi, \psi \in  {\rm PCont}_+(M, \xi)$, we have that $d_{\geq_o}(\phi, \psi) \geq d_{\geq_s}(\sigma(\phi), \sigma(\psi))$. Thus we get the desired conclusion. \end{proof}

\begin{remark} The normalization condition that $\F|_{t=0} = \k_{\Delta}$ for every $\F \in \D_{\rm adm}^b(\k_{X \times X \times I})$ is crucial to the proof of Proposition \ref{prop-sheaf-order}, in particular, to the proof of the anti-symmetric property of the relation $\geq_s$. \end{remark}

\section{Appendix} \label{sec-app}
In this appendix, we provide another formulation of the relative growth rate $\gamma_{\geq}(\cdot, \cdot)$ on $G_+$ which is defined via $\rho_{\geq}^+(\cdot, \cdot)$ in (\ref{dfn-rgr}). In fact, we will give an equivalent formulation of $\rho_{\geq}^+(\cdot, \cdot)$ on $G_+$ which claims to be defined by only prime numbers. Recall that $G_+$ is the set of all the dominants of $G$, and without loss of generality assume that $\rho_{\geq}^+(a,b)>0$ for any $a, b \in G_+$. For any $a, b \in G_+$, we call the pair $(k,l) \in \N^2$ an {\it $(a, b)$-ordering pair} if it satisfies $a^k \geq b^l$. Denote by $\mathcal O_{(a,b)}$ the set of all the $(a, b)$-ordering pairs. Recall that 
\[ \rho^{+}_{\geq} (a,b) = \inf\left\{ \frac{k}{l} \in \Q_+ \,\bigg| \, (k,l) \in \mathcal O_{(a,b)} \right\}. \]
Denote by $\mathcal P_{(a,b)}$ the subset of $\mathcal O_{(a,b)}$ which consists of all the $(a,b)$-ordering pairs $(p,q) \in \mathcal O_{(a,b)}$ such that both $p$ and $q$ are prime numbers. Then we have the following property.

\begin{prop} \label{rgr-prime} For any $a,b \in G_+$, we have \begin{equation} \label{equ-rgr-prime}
\rho^{+}_{\geq}(a,b) =\inf\left\{ \frac{p}{q} \in \Q_+\,\bigg| \, (p,q) \in \mathcal P_{(a,b)} \right\}. 
\end{equation}
\end{prop}

\begin{proof} Denote the left-hand side of the equation (\ref{equ-rgr-prime}) by $L$ and the right-hand side of the equation (\ref{equ-rgr-prime}) by $R$. Since $\mathcal P_{(a,b)} \subset \mathcal O_{(a,b)}$, it is obvious that $L\leq R$. For the other direction, by definition $L =\rho_{\geq}^+(a,b)$ implies there exists a sequence of pairs $\{(k_n, l_n)\}_{n \in \N}$ in $\mathcal O_{(a,b)}$ such that $\lim_{n \to\infty} \frac{k_n}{l_n} =L$. Passing to a subsequence if it is necessary, we can assume that $l_n \to \infty$. Since $L > 0$, we know that $k_n \to \infty$ as well. Pick $\alpha = 0.6$, then by Baker-Harman-Pintz's theorem in \cite{BHP01}, when $n$ is sufficiently large, there exists a prime number, denoted by $q_n$ lying in $[l_n - l_n^{\alpha}, l_n]$. In particular, 
\begin{equation}\label{est-1}
q_n \leq l_n \,\,\,\,\mbox{and}\,\,\,\, \lim_{n \to \infty} \frac{q_n - l_n}{l_n} = 0.
\end{equation} 
Similarly, when $n$ is sufficiently large, there exists a prime number, denoted by $p_n$ lying in $[k_n, k_n + k_n^{\alpha}]$. In particular, 
\begin{equation}\label{est-2}
p_n \geq k_n \,\,\,\,\mbox{and}\,\,\,\, \lim_{n \to \infty} \frac{p_n - k_n}{k_n} = 0.
\end{equation} 
Each $(k_n, l_n)$ belongs to $\mathcal O_{(a,b)}$, so $a^{p_n} \geq a^{k_n} \geq b^{l_n} \geq b^{q_n}$. Therefore, $(p_n, q_n) \in \mathcal P_{(a, b)}$. Moreover, due to relations (\ref{est-1}) and (\ref{est-2}), 
\begin{align*}
\lim_{n \to \infty} \frac{p_n}{q_n} = \lim_{n \to \infty} \frac{p_n - k_n + k_n}{q_n - l_n + l_n} =\lim_{n \to \infty} \frac{\left(\frac{p_n - k_n}{k_n} +1 \right) k_n}{\left(\frac{q_n - l_n}{l_n} +1 \right) l_n} = \lim_{n\to \infty} \frac{k_n}{l_n} = L.
\end{align*}
This implies that $R \leq L$, which is the desired conclusion. \end{proof}

\begin{remark} The only place in the proof of Lemma \ref{rgr-prime} where we use the hypothesis that $L = \rho_{\geq}^+(a, b)> 0$ is to guarantee that the sequence $\{k_n\}_{n \in \N}$ tends to infinity (so that we can apply Baker-Harman-Pintz's theorem from \cite{BHP01} to obtain an $(a,b)$-ordering pair in $\mathcal P_{(a,b)}$). 
\end{remark}

\section{Acknowledgement} This project was motivated by various questions raised in a guided reading course at Tel Aviv University in Fall 2018, which was organized by Leonid Polterovich. We thank Leonid Polterovich for providing this valuable opportunity, as well as many interesting and enlightening questions. We also thank Matthias Meiwes, Egor Shelukhin, and Michael Usher for helpful conversations. In particular, we thank Leonid Polterovich for useful feedback on the first version of the paper. This work was completed when the second author holds the CRM-ISM Postdoctoral Research Fellow at CRM, University of Montreal, and the second author thanks this institute for its warm hospitality.

\noindent\\

\bibliographystyle{amsplain}
\bibliography{biblio_cbm}
\noindent\\

\end{document}